\documentclass[10pt,reqno]{amsart}
\usepackage{fullpage}
\usepackage[toc,page,title,titletoc,header]{appendix}
\usepackage{amsthm,amsmath,amsfonts,amssymb,euscript,hyperref,graphics,color,slashed,mathrsfs,graphicx,mathtools,enumerate,cite}
\hypersetup{colorlinks, citecolor=black, urlcolor=black, linkcolor=black}
\usepackage[english]{babel}
\usepackage{bm}

\allowdisplaybreaks

\numberwithin{equation}{section}

\newcommand{\clambda}{\mathring{\lambda}}
\newcommand{\cE}{\mathring{E}}
\newcommand{\cF}{\mathring{F}}
\newcommand{\cmE}{\mathring{\mathcal{E}}}

\newtheorem{theorem}{Theorem}[section]
\newtheorem{lemma}[theorem]{Lemma}

\newtheorem{remark}[theorem]{Remark}

\title[Lifespan estimate for MHD]{The estimate of lifespan for ideal incompressible MHD systems}

\author[Mengni Li]{Mengni Li}
\address{Southeast University}
\email{krisymengni@gmail.com}

\author[Hongwei Yuan]{Hongwei Yuan}
\address{University of Macau}
\email{hwyuan@um.edu.mo}

\begin{document}
\begin{abstract}
This paper concerns the lifespan  estimate for three-dimensional ideal incompressible MHD systems. By introducing  $L^{3/2}$-type energy, we lower the regularity assumption for the initial data  as well as  the order of the highest derivatives involved in the total weighted energy estimate. The new weights proposed in this paper are symmetric, unified, and compatible with both $L^2$-type energy and $L^{3/2}$-type energy.

\noindent
\textbf{Running title}:  Lifespan estimate for MHD
\end{abstract}

\maketitle

\section{Introduction}  

Let $b(t,x):\mathbb{R}\times\mathbb{R}^3\to\mathbb{R}^3$ denote the magnetic field, $v(t,x):\mathbb{R}\times\mathbb{R}^3\to\mathbb{R}^3$  the fluid velocity, and $p(t,x):\mathbb{R}\times\mathbb{R}^3 \to\mathbb{R}$ the fluid pressure. Set both the fluid viscosity and the magnetic viscosity be zero. Then the governing system for three-dimensional ideal incompressible magnetohydrodynamics (MHD)  in the whole space can be formulated as
\begin{equation}\label{MHD general}
	\begin{cases}
		&\partial_t  v+ v\cdot \nabla v = -\nabla p + (\nabla\times b)\times b, \\
		&\partial_t b + v\cdot \nabla b =  b \cdot \nabla v,\\
		&\operatorname{div} v =0,\\
		&\operatorname{div} b =0.
\end{cases}\end{equation}

The first equation in \eqref{MHD general} is the momentum equation derived from the Navier-Stokes equations in the presence of the Lorentz force:
\[\underbrace{(\nabla\times b)\times b}_{\text{Lorentz force}}=\underbrace{\nabla(-\tfrac{1}{2}| b|^2)}_{\sim \text{ magnetic pressure}}+\underbrace{\ b \cdot \nabla  b\ }_{\sim \text{ magnetic tension}}.\]
Here, two distinct types of restoring forces arise,  leading to different wave phenomena:  magnetic pressure  generates magnetoacoustic waves, 
which are similar to acoustic waves produced by ordinary fluid pressure; 
in contrast, magnetic tension  creates Alfv\'en waves, 
a wave phenomenon 
discovered by the Swedish Nobel laureate H. Alfv\'en \cite{Alfven}. 
By still using $p$ to denote the total pressure $p+\tfrac{1}{2}| b|^2$, we can rephrase the momentum equation as:
\begin{equation*}
	\partial_t  v+ v\cdot \nabla v = -\nabla p+ b \cdot \nabla b.
\end{equation*} 
In \eqref{MHD general}, 
the second equation  is the induction equation for the magnetic field derived from Maxwell's equations. The condition $\operatorname{div}v=0$ ensures that the fluid is incompressible, while  $\operatorname{div}b=0$ is the Gauss condition, indicating that the magnetic field is solenoidal (i.e., it has no sources or sinks).

In an ideal scenario with a perfectly conducting fluid, magnetic field lines tend to remain frozen within the fluid.  This motivates us to analyze the situation where a fluid flows through a steady strong magnetic field $B_0$,  while a small velocity field $v$ acts perpendicular to $B_0$. We let
\begin{equation}\label{condition-B0}
B_0=(0,0,1)
\end{equation}  
to denote a strong magnetic background. It is easy to see that  $(v,b)=(0,B_0)$ is a stationary solution to the system \eqref{MHD general}. 
 Using the perturbation variables 
$$z_+=(v-0)+(b-B_0),\ \ \ \ z_-=(v-0)-(b-B_0),$$ we can reexpress the system  \eqref{MHD general} as
\begin{equation}\label{eq:MHD}\begin{cases}
		&\partial_{t}z_{+}+(z_--B_0)\cdot \nabla z_{+} =-\nabla p,\\
		&\partial_{t}z_{-}+(z_++B_0)\cdot \nabla z_{-} =-\nabla p,\\
		&\operatorname{div} z_{+}=0,\\
		&\operatorname{div} z_{-}=0.
\end{cases}\end{equation}
Generally speaking,   small initial perturbations $(v,b)$ around $(0,B_0)$  generate two families of Alfv\'en waves, $z_+$ and $z_-$, which propagate along the directions of $B_0$ and $-B_0$, respectively.

   In this scenario, we define two characteristic 	vector fields $L_{+}$ and $L_{-}$ as
   \begin{equation}\label{eq:L-vector}
   	L_{\pm}=T+ (z_\pm \pm B_0),
   \end{equation}
   where $T$ is the normal vector field of the time slice $\Sigma_{t_0}=\{(t,x)\in\mathbb{R}\times\mathbb{R}^3|t=t_0\}$ for any $t_0\in\mathbb{R}$. For convenience, we also use $L_+$ and $L_-$ to denote the partial differential operators that appear in the system \eqref{eq:MHD}, namely
   \begin{equation}\label{eq:L-operator}
   	L_{\pm}=\partial_t+ (z_\pm \pm B_0)\cdot\nabla.
   \end{equation}
   This minor abuse of notation will simplify the expression.  
   Here we note that 
   $L_\pm\sim (\partial_t\pm \partial_{x_3}$) due to 
   \eqref{condition-B0}.

   Two characteristic functions $u_+=u_+(t,x)$ and $u_-=u_-(t,x)$ are then defined as 
   \begin{equation}\label{definitionforu}
   	\begin{cases}
   		&L_\pm u_\pm = 0,\\
   		&u_\pm\big|_{\Sigma_0} = x_3.
   	\end{cases}
   \end{equation} 
   We see that $u_\pm
   \sim (x_3\mp t)$. 
   For given real numbers $u_{+,0}$, $u_{-,0}$ and $t_0$, we define 
   the following characteristic hypersurfaces: 
   \begin{align*}
   	&C^{+,t_0}_{u_{+,0}}=\big\{(t,x)\in [0,t_0]\times \mathbb{R}^3\big|u_+(t,x) =u_{+,0}\big\},\\
   	&C^{-,t_0}_{u_{-,0}}=\big\{(t,x)\in [0,t_0]\times \mathbb{R}^3\big|u_-(t,x) =u_{-,0}\big\}.
   \end{align*}
   We note that the vector fields $L_+$ and $L_-$ are  parallel to $C^{+,t_0}_{u_{+,0}}$ and $C^{-,t_0}_{u_{-,0}}$, respectively.

We turn to briefly review some progress related to the study of three-dimensional incompressible magnetohydrodynamic (MHD) systems in the whole space, particularly focusing on small data theory in the presence of strong magnetic backgrounds. 
This topic  began at least with C. Bardos, C. Sulem and P.-L. Sulem in their notable work \cite{Bardos-Sulem-Sulem}, where they constructed global solutions  in H\"older spaces for the ideal case by using convolutions with fundamental solutions. 
In a series of studies, including \cite{Lin-Zhang,Lin-Xu-Zhang,Xu-Zhang}, 
F. Lin, L. Xu and P. Zhang  later developed global existence results in  energy spaces  for  viscous cases by employing the Fourier method.  
The reliance of their data on viscosity provides room for possible improvement. 
Shortly afterwards, three research groups addressed this issue nearly at the same time: L.-B. He, L. Xu, and P. Yu contributed to \cite{He-Xu-Yu}; Y. Cai and Z. Lei  to \cite{Cai-Lei}; and D. Wei and Z. Zhang  to \cite{Wei-Zhang-2017, Wei-Zhang-2018}.  Based on the inherent null structure of nonlinear terms in strong magnetic backgrounds, 
these three groups resourcefully utilized various energy methods   to ensure that data independence from viscosity could also lead to global existence results for both ideal and viscous cases.

In recent studies of rigidity from infinity \cite{Li-Yu} and inverse scattering \cite{Li}, the proof of global existence was enhanced through various adaptations of the weighted energy method from \cite{He-Xu-Yu}. Although the works \cite{He-Xu-Yu, Li-Yu, Li} focused on the global existence of solution, they did not address the lowest possible order of the highest derivatives needed to close the energy estimates. They established that the highest order of derivatives in use is at least six, denoted by $N_*+1\geq 6$.  Therefore, the highest order of the total energy estimates in these works is also at least six. In fact, if we consider the lowest possible order of the highest derivatives used in the total energy estimates, it can be  demonstrated as four using the same proofs   (i.e. $N_*+1=4$, following the notation in \cite{Li}). We are now interested in whether it is possible to establish energy estimates that remain closed at a lower order when initial data perturbations are assumed to have lower regularity. 
To this end, we follow the use of the weighted energy method (with a new type of energy to be introduced later), and relax the requirement of proving the existence of global solutions to demonstrating the lifespan estimate for solutions with lower regularity.

The local existence and lifespan of solution is also a topic of significant attention. A number of results in different cases of MHD have been achieved by many researchers using various methods. For example, in \cite{Fefferman-Mc-R-R-2014, Fefferman-Mc-R-R-2018}, C. Fefferman, D. McCormick, J. Robinson, and J. Rodrigo established local existence for  non-resistive MHD equations in two or three space dimensions with different initial data. Specifically,   \cite{Fefferman-Mc-R-R-2014} is based on higher-order commutator estimates, whereas  \cite{Fefferman-Mc-R-R-2018} relies on maximal regularity estimates for the Stokes equation. By using the cancellation mechanism, F. Xie and T. Yang \cite{Xie-Yang} showed a lower bound of the lifespan of solution to the two-dimensional MHD boundary layer system with analytic perturbation of general shear flow. After some suitable  randomization of  initial data, 
L. Du and T. Zhang \cite{Du-Zhang} obtained the almost sure existence of global weak solutions for incompressible MHD systems in periodic space domains, and 
M. Dai \cite{Dai} obtained a similar result for the supercritical electron MHD model. 
For some latest progress on this topic in the  ideal case, D. Cobb and F. Fanelli \cite{Cobb-Fanelli}
showed an improved lifespan in 
two space dimensions 
by utilizing the Elsässer form of ideal MHD rewritten in vorticity formulation. More recently, S. Liu and Z. Xin \cite{Liu-Xin-1} proved the local wellposedness of general plasma–vacuum interface problems for ideal incompressible MHD equations.

The principal goal of this paper is to construct the local existence and lifespan estimate of solution for ideal MHD.  Moreover, we  find that the regularity order of initial data perturbations can be further reduced by one, which allows the lowest possible order of the highest derivatives used in the total energy estimates to be lowered to three  (i.e. $N_*+1=3$, following the notation in \cite{Li}). This reduction is primarily based on the observation that,  when performing energy estimates up to the third order, it is more effective to close the estimates by applying the  Gagliardo-Nirenberg-Sobolev inequality and introducing a new $L^{3/2}$-type energy of the third-order derivative, rather than using the  Sobolev inequality and  $L^2$-type  energy of the fourth-order derivative as described in \cite{Li}. 
This new approach represents one of the highlights of this paper.

In view of this new 
 $L^{3/2}$-type energy of the third-order derivatives, an important  additional task is to identify  suitable weights to  align with it, as well as to  harmonize with all $L^2$-type energies of derivatives from zeroth to third order. 
In this process, we consider   weight functions $\lambda_{\pm}=\lambda_{\pm}(u_\mp)$ defined on $[0,t^*]\times\mathbb{R}^3$
with the properties
\begin{equation}\label{eq:lambda-condition}
	L_{\pm}\lambda_{\pm}\lesssim \lambda_{\pm},\ \ \ \ 	L_{\mp}\lambda_{\pm}=0,
\end{equation} 
where the operator  definition  \eqref{eq:L-operator} is used for $L_\mp$. Then we adopt 
 the method of undetermined coefficients, letting the  weight corresponding to $z_\pm$  be
   $$\lambda=\lambda_+^{a_+}\lambda_-^{a_-}$$ with constants $a_+$ and $a_-$ to be determined. The values of $a_+$ and $a_-$ 
  will be   determined by how to fully utilize the symmetry of the two families of waves $z_+$ and $z_-$, as well as  the properties outlined in Lemma \ref{lemma:weights} that a suitable weight should satisfy.  Finally we find that $$a_+=a_-=2.$$ 
This new choice of weights is also one of the highlights of this paper.

Now we are ready to define the
$L^2$-type zero-order weighted energy norm (through $\Sigma_t$) and  flux norm (through $C^{\pm,t}_{u_{\pm}}$)  as 
\begin{align*}
	&E_{\pm,2}(t)=\left\|\left(\lambda_\pm^2\lambda_\mp^2z_\pm\right)(t,\cdot)\right\|^2_{L^2_x}=\int_{\Sigma_t}\left(\lambda_\pm^4\lambda_\mp^4|z_{\pm}|^2\right)(t,x)\, dx,\\ &F_{\pm,2}(t)=\sup_{u_{\pm}\in\mathbb{R}}\int_{C^{\pm,t}_{u_{\pm}}}\left(\lambda_\pm^4\lambda_\mp^4|z_{\pm}|^2\right)(\tau,u_\mp)\, d\sigma_{\pm},
\end{align*}
where $d\sigma_{\pm}:=d\sigma_{\pm}(\tau,u_\mp)$ represents the induced surface measure of the characteristic hypersurface $C_{u_\pm}^{\pm,t}$.
Given a multi-index $\alpha=(\alpha_1,\alpha_2,\alpha_3)$ of order $k=|\alpha|=\alpha_1+\alpha_2+\alpha_3$ with $\alpha_i \in \mathbb{Z}_{\geq 0}$ ($i=1,2,3$), we consider $\nabla^{k}$ as  $\partial^{\alpha_1}_{x_1}\partial^{\alpha_2}_{x_2}\partial^{\alpha_3}_{x_3}$ and denote $f^{(\alpha)}=f^{[k]}=\nabla^{k} f$.
For any $|\alpha|\leq 2$, we define the  $L^2$-type higher-order weighted energy and flux norms as 
\begin{align*}
	&E^{(\alpha)}_{\pm,2}(t)=\left\|\left(\lambda_\pm^2\lambda_\mp^2j^{(\alpha)}_\pm\right)(t,\cdot)\right\|^{2}_{ L^{2}_x}=\int_{\Sigma_t}\left(\lambda_\pm^4\lambda_\mp^4\big|{j^{(\alpha)}_{\pm}}\big|^2\right)(t,x)dx,\ \ \ \ E^{[k]}_{\pm,2}(t)= \sum_{|\alpha|=k}E^{(\alpha)}_{\pm,2}(t),\\
	&F^{(\alpha)}_{\pm,2}(t)=\sup_{u_{\pm}\in\mathbb{R}}\int_{C^{\pm,t}_{u_{\pm}}}\left(\lambda_\pm^4\lambda_\mp^4\big|j^{(\alpha)}_{\pm}\big|^2\right)(\tau,u_\mp)\, d\sigma_\pm,\ \ \ \ F^{[k]}_{\pm,2}(t)=\sum_{|\alpha|=k}F^{(\alpha)}_{\pm,2}(t),
\end{align*}
where $j_\pm=\operatorname{curl}z_\pm$. 
For the highest-order derivative that we use, we also define the following $L^{3/2}$-type weighted energy norm:
\begin{equation*}
	E^{(\alpha)}_{\pm,3/2}(t)=\left\|\left(\lambda_\pm^2\lambda_\mp^2j^{(\alpha)}_\pm\right)(t,\cdot)\right\|^{3/2}_{ L^{3/2}_x}=\int_{\Sigma_t}\left(\lambda_\pm^3\lambda_\mp^3\big|{j^{(\alpha)}_{\pm}}\big|^{3/2}\right)(t,x)dx,\ \ \ \ E^{[2]}_{\pm,3/2}(t)= \sum_{|\alpha|=2}E^{(\alpha)}_{\pm,3/2}(t).
\end{equation*}

Before stating the main result, let us give a  rescaled version of these notations. We define the rescaled weight function
\begin{equation*}
	\clambda_\pm=\left(1+|u_\mp|^2\right)^{\frac{1+\delta}{6}},
\end{equation*}
energy norms
\begin{align*}
	&\cE_{\pm,2}(t)=\left\|\left(\clambda_\pm^2\clambda_\mp^2z_\pm\right)(t,\cdot)\right\|^2_{L^2_x},\\ 	&\cE^{(\alpha)}_{\pm,2}(t)=\left\|\left(\clambda_\pm^2\clambda_\mp^2j^{(\alpha)}_\pm\right)(t,\cdot)\right\|^{2}_{ L^{2}_x},\ \ \ \ \cE^{[k]}_{\pm,2}(t)= \sum_{|\alpha|=k}\cE^{(\alpha)}_{\pm,2}(t),\\
	&	\cE^{(\alpha)}_{\pm,3/2}(t)=\left\|\left(\clambda_\pm^2\clambda_\mp^2j^{(\alpha)}_\pm\right)(t,\cdot)\right\|^{3/2}_{ L^{3/2}_x},\ \ \ \ \cE^{[2]}_{\pm,3/2}(t)= \sum_{|\alpha|=2}\cE^{(\alpha)}_{\pm,3/2}(t),
\end{align*}
and flux norms
\begin{align*}
	&\cF_{\pm,2}(t)=\sup_{u_{\pm}\in\mathbb{R}}\int_{C^{\pm,t}_{u_{\pm}}}\left(\clambda_\pm^4\clambda_\mp^4|z_{\pm}|^2\right)(\tau,u_\mp)\, d\sigma_{\pm},\\
	&\cF^{(\alpha)}_{\pm,2}(t)=\sup_{u_{\pm}\in\mathbb{R}}\int_{C^{\pm,t}_{u_{\pm}}}\left(\clambda_\pm^4\clambda_\mp^4\big|j^{(\alpha)}_{\pm}\big|^2\right)(\tau,u_\mp)\, d\sigma_\pm,\ \ \ \ \cF^{[k]}_{\pm,2}(t)=\sum_{|\alpha|=k}\cF^{(\alpha)}_{\pm,2}(t).
\end{align*}

The energies denoted by $\cE$ and $\cF$ are defined using the
fixed weights
$\clambda_\pm$
and will be the quantities appearing in the statement of our main
theorem. In the proof, we will use an  $R$-dependent version 
weights $\lambda_\pm$  as in \eqref{eq:R-dep} 
and the corresponding energy  $E$ and flux 
$F$ to derive the a priori estimates. (By using the parameter $R$, the derivatives of
$\lambda_\pm$ along the characteristic directions gain a factor
of order $R^{-1}$. The choice of $R$ will  be given 
at the end of the bootstrap argument.)

The main result of this work is the following small weighted energy estimates in $\mathbb{R}^3$. We remark that, by incorporating together the $L^2$-type energies of derivatives from zeroth to third order and the 
$L^{3/2}$-type energy of the third-order derivatives, the total energy can be improved to some small power of the  total energy at initial time.

\begin{theorem}
	\label{thm}  
	Let $\delta \in(0,\tfrac{1}{4})$.
There exists a small constant $\varepsilon_0\in(0,1)$ such that for any $\varepsilon\in (0,\varepsilon_0]$, if the initial data $\big(z_+(0,x),z_-(0,x)\big)$ of \eqref{eq:MHD} satisfy
	\begin{equation}\label{eq:initial energy}
		\cmE(0):=	\sum_{+,-}\bigg(\cE_{\pm,2}+\sum_{k=0}^{2}\cE^{[k]}_{\pm,2}+\cE^{[2]}_{\pm,3/2}\bigg)(0) 
		\leq \varepsilon,
	\end{equation}
	then the ideal MHD system \eqref{eq:MHD}  admits a unique solution $\big(z_+(t,x),z_-(t,x)\big)$ in  $[0,T]$ with
	\begin{equation}\label{eq:estimate-T}
		T=O(\varepsilon^{-\frac{3}{8(1+\delta)}(1-\kappa)})
	\end{equation} for any constant $\kappa\in (0,1)$.
	Moreover,  the solution $\big(z_+(t,x),z_-(t,x)\big)$ enjoys the following small weighted energy estimate for all $t\in[0,T]$:
	\begin{equation}\label{eq:energyestimate} 
		 \cmE(t):=\sum_{+,-}\bigg(\cE_{\pm,2}+\cF_{\pm,2}+\sum_{k=0}^{2}\cE^{[k]}_{\pm,2}+\sum_{k=0}^{2}\cF^{[k]}_{\pm,2}+\cE^{[2]}_{\pm,3/2}\bigg)(t)\leq\varepsilon^\kappa.
	\end{equation}
\end{theorem}

\begin{remark}
 		Theorem~\ref{thm} shows that sufficiently small initial data
 		give rise to  solutions on a polynomially long time
 		interval. In particular, as the size of the initial data tends to
 		zero, the lifespan grows polynomially and therefore becomes
 		arbitrarily large.
 \end{remark}

The rest of this paper is arranged as follows. 
In Section  \ref{sec:pre},  we provide several technical estimates which will be frequently used in this paper. In Section \ref{sec:linear}, we establish two types of linear energy estimates with unified weights: one is $L^2$-type   and the other is  $L^{3/2}$-type. Finally, we present new weighted energy estimates order by order to close the estimates, and prove Theorem \ref{thm} in Section \ref{sec:energy}.

\section{Bootstrap assumption and technical estimates}\label{sec:pre}
 From now on, we take
\begin{equation}\label{eq:R-dep}
\lambda_\pm=\left(R^2+|u_\mp|^2\right)^{\frac{1+\delta}{6}}\leq R^{\frac{1+\delta}{3}}\clambda_\pm,
\end{equation}
where $R>1$ can be a large constant,  
and $\delta>0$ is a small constant 
which will be  determined by \eqref{def-delta}.
 These weight functions  clearly satisfy \eqref{eq:lambda-condition}. More precisely, one can verify that 
 \begin{equation}\label{eq:lambda-condition'}
 	L_{\pm}\lambda_{\pm}\lesssim \frac{|u_\mp|}{R^2+|u_\mp|^2}\lambda_{\pm}\lesssim \frac{1}{R}\lambda_{\pm},\ \ \ \ 	L_{\mp}\lambda_{\pm}=0,
 \end{equation} 
By virtue of \eqref{eq:initial energy},  the  total energy at initial time satisfies
\begin{equation}\label{eq_2_2} 
		\cmE(0)\leq \mathcal{E}(0):=	\sum_{+,-}\bigg(E_{\pm,2}+\sum_{k=0}^{2}E^{[k]}_{\pm,2}+E^{[2]}_{\pm,3/2}\bigg)(0)\leq R^{\frac{8(1+\delta)}{3}}\cmE(0)
		\leq R^{\frac{8(1+\delta)}{3}}\varepsilon.
	\end{equation}

 \subsection{Bootstrap assumption and goal}
The proof of our main theorem is based on 	the method of continuity. At the beginning,  
  we assume that there exists a $T_1\in(0,t^*]$ such that the following holds:
 	\begin{equation}\label{Bootstrap on energy}
 		\mathcal{E}(t):=\sum_{+,-}\bigg(E_{\pm,2}+F_{\pm,2}+\sum_{k=0}^{2}E^{[k]}_{\pm,2}+\sum_{k=0}^{2}F^{[k]}_{\pm,2}+E^{[2]}_{\pm,\frac{3}{2}}\bigg)(t)\leq 2\varepsilon^{\kappa},\ \ \forall t\in [0,T_1].
 	\end{equation}
This assumption is legitimate as a result of \eqref{eq_2_2} and local existence of \eqref{eq:MHD}. 
 To close the continuity argument, we need to show that, for any fixed small constant $\varepsilon_0\in(0,1)$, there exists a constant  $T=O(\varepsilon^{-\frac{3}{8(1+\delta)}(1-\kappa)})$ such that whenever $0<T_1\leq T$,  the bound in \eqref{Bootstrap on energy} can be improved as follows:
 	\begin{equation}\label{Bootstrap on energy improved}
 		\mathcal{E}(t)=\sum_{+,-}\bigg(E_{\pm,2}+F_{\pm,2}+\sum_{k=0}^{2}E^{[k]}_{\pm,2}+\sum_{k=0}^{2}F^{[k]}_{\pm,2}+E^{[2]}_{\pm,\frac{3}{2}}\bigg)(t)\leq \varepsilon^{\kappa},\ \ \forall t\in [0,T_1].
 	\end{equation} 
 Hence the proof of our main theorem can be reduced to proving \eqref{Bootstrap on energy improved} under the assumption \eqref{Bootstrap on energy}. We will leave this proof to the next sections.

\subsection{Technical estimates}

 In what follows, the notations $A\lesssim  B$, $A\gtrsim B$, and $A\sim B$ mean that $A\leq CB$,  $A\geq CB$, and $C^{-1}B\leq A\leq CB$ respectively for some universal constant $C>0$.

As a direct consequence of weighted div-curl lemma in \cite{He-Xu-Yu}, we have 
 \begin{lemma}\label{lemma:divcurl}
 		Let $\lambda(x)\geq 1$ 
 	be a smooth positive function on $\mathbb{R}^3$ with the additional property $|\nabla\lambda|\lesssim \lambda$. Then  for any $0\leq k\leq 2$,  it holds in $\mathbb{R}^3$ that 
 	\begin{equation}\label{eq:z to j inequality}
 		\big\|\sqrt\lambda\nabla z_\pm^{[k]}\big\|_{L^2_x}^2 \lesssim \big\|\sqrt\lambda
 		z_\pm\big\|_{L^2_x}^2+\sum_{l=0}^{k} \big\|\sqrt\lambda
 		j_\pm^{[l]}\big\|_{L^2_x}^2.
 	\end{equation}
 \end{lemma}

Some properties of weight functions are collected  in the following lemma. We omit the proofs of these properties since they are similar to Lemma 3.5 and Lemma 3.7 in \cite{Li}.

\begin{lemma}[Weight properties]\label{lemma:weights}
We have the following inequalities:
	\begin{enumerate}[(i)]
		\item for $l=1,2$, there holds
		\begin{equation}\label{differentiate weights coro}
			\begin{split}
					&|\nabla^l \left(\lambda_\pm\lambda_\mp\right)|\lesssim \frac{1}{R}\lambda_\pm\lambda_\mp, \\
				&|\nabla^l \left(\lambda_\pm^2\lambda_\mp^2\right)|\lesssim \frac{1}{R}\lambda_\pm^2\lambda_\mp^2.
			\end{split}
		\end{equation}
	
		\item there holds
		\begin{equation}\label{eq:integral}
			\int_{\mathbb{R}}\frac{1}{\lambda^2_\pm\lambda^2_\mp}\, du_\pm\lesssim\int_{\mathbb{R}}\frac{1}{\lambda^2_\mp}\, du_\pm\lesssim 1.
		\end{equation}
		\item there holds 
		\begin{equation}\label{eq:xgeqleq}
		\left(\lambda_\pm^3\lambda_\mp^3\right)(t,x)\lesssim\begin{cases}
				\left(\lambda_\pm^3\lambda_\mp^3\right)(t,x')+{|x-x'|^{2(1+\delta)}},&\ \ \ \ \text{if}\ \  |x-x'|\geq 1\\
			\left(\lambda_\pm^3\lambda_\mp^3\right)(t,x'),&\ \ \ \ \text{if}\ \  |x-x'|\leq 2,
			\end{cases}
		\end{equation}
		and both the inequalities in \eqref{eq:xgeqleq} hold when $1\leq |x-x'|\leq 2$.
		\item 	 we have the  separation property as 
		\begin{equation}\label{eq:separation weight}
	\left(\lambda_\pm^4\lambda_\mp^4\right)(t,\cdot)\gtrsim (R+t)^{\frac{4(1+\delta)}{3}},
		\end{equation}
	and hence  
	\begin{equation}
		\int_{0}^{t}\frac{1}{\lambda_\pm^4\lambda_\mp^4}
		d\tau\lesssim\int_{0}^{t}\frac{1}{(R+\tau)^{\frac{4(1+\delta)}{3}}}d\tau\lesssim 1.
	\end{equation}
	\end{enumerate}	
\end{lemma}

\begin{proof} 
 For (i), by the definition of $\lambda_\pm$ and the chain rule for differention, for any $k\in\{1,2,3\}$, we have
\begin{equation*}
	\frac{\partial \lambda_\pm}{\partial x_k}
=\frac{1+\delta}{3}\frac{u_\mp\frac{\partial u_\mp}{\partial x_k}}{R^2+|u_\mp|^2}\lambda_\pm
\end{equation*}
Noting that   $|\frac{\partial u_\mp}{\partial x_k}|\leq 2$ (see Appendix \ref{sec:flow}), we obtain  $\nabla \lambda_+\leq \frac{2}{R}\lambda_+$. Hence \eqref{differentiate weights coro} is proved. 

	(ii) and (iii) can be proved by direct computations, and (iv) can be given by Appendix \ref{sec:weight}. 
\end{proof}

   Three frequently used tools are the following 
   variants of Gagliardo-Nirenberg-Sobolev inequality ($\|f\|_{L_x^6(\mathbb{R}^3)}\lesssim \|\nabla f\|_{L_x^2(\mathbb{R}^3)}$),   Sobolev inequality ($\|f\|_{L_x^\infty(\mathbb{R}^3)}\lesssim \|
   f\|_{H_x^2(\mathbb{R}^3)}$) and spacetime estimates with various 
   weights.

   	\begin{lemma}[Weighted Gagliardo-Nirenberg-Sobolev inequality]\label{lemma:GNS}
   		In $\mathbb{R}^3$,	for all $k\leq 1$, we have 
   		\begin{align*}
   			&\left\|\left(\lambda_\pm\lambda_\mp\nabla z^{[k]}_\pm\right)(t,\cdot)\right\|_{ L^6_x} \lesssim\left\|\left(\frac{1}{R}\lambda_\pm\lambda_\mp\nabla z^{[k]}_\pm\right)(t,\cdot)\right\|_{ L^2_x}+\left\|\left(\lambda_\pm\lambda_\mp\nabla z^{[k+1]}_\pm\right)(t,\cdot)\right\|_{ L^2_x}
   			\lesssim    \mathcal{E}(t)^{1/2}.
   		\end{align*}
   	\end{lemma}
   	
   \begin{proof}
   We can apply the standard Gagliardo-Nirenberg-Sobolev inequality on $\mathbb{R}^3$ and weight property 
   	\eqref{differentiate weights coro}   to derive that 
   		\begin{align*}
   		\left\|\left(\lambda_\pm\lambda_\mp\nabla z^{[k]}_\pm\right)(t,\cdot)\right\|_{ L^6_x}&\lesssim \left\|\nabla\big(\lambda_\pm\lambda_\mp\nabla z^{[k]}_\pm\big)(t,\cdot)\right\|_{L^2_x}
   		\\   	
   		& \lesssim\left\|\left(\frac{1}{R}\lambda_\pm\lambda_\mp\nabla z^{[k]}_\pm\right)(t,\cdot)\right\|_{ L^2_x}+\left\|\left(\lambda_\pm\lambda_\mp\nabla z^{[k+1]}_\pm\right)(t,\cdot)\right\|_{ L^2_x}\\
   		&\lesssim\left(\frac{1}{R} +1 \right)\mathcal{E}(t)^{1/2},
   	\end{align*}
   which gives the desired result. 
   \end{proof}

 \begin{lemma}[Weighted Sobolev inequality in space]\label{lemma:Sobolev}
 		In $\mathbb{R}^3$,  we have 
 	\begin{equation*} 
 		\begin{split}
 			&\left\|\left(\lambda_\pm^2\lambda_\mp^2z_{\pm}\right)(t,\cdot)\right\|_{L^\infty_x}
 			\lesssim \left\|\left(\lambda_\pm^2\lambda_\mp^2z_{\pm}\right)(t,\cdot)\right\|_{H^2_x}
 			\lesssim  \mathcal{E}(t)^{1/2},\\
 			&\left\|\left(\lambda_\pm^2\lambda_\mp^2j_{\pm}\right)(t,\cdot)\right\|_{L^\infty_x}
 			\lesssim \left\|\left(\lambda_\pm^2\lambda_\mp^2j_{\pm}\right)(t,\cdot)\right\|_{H^2_x}
 			\lesssim  \mathcal{E}(t)^{1/2},\\ 
 			&	\left\|\left(\lambda_\pm^2\lambda_\mp^2\nabla z_{\pm}\right)(t,\cdot)\right\|_{L^\infty_x}\lesssim  \left\|\left(\lambda_\pm^2\lambda_\mp^2\nabla z_{\pm}\right)(t,\cdot)\right\|_{H^2_x}\lesssim  \mathcal{E}(t)^{1/2}. 
 	\end{split}\end{equation*}
\end{lemma}
\begin{proof}
The proof is based on the standard Sobolev inequality on $\mathbb{R}^3$, weight property 
\eqref{differentiate weights coro} and div-curl Lemma \ref{lemma:divcurl}. We first prove the inequality on $j_\pm$:
\begin{align*}
	\left\|\left(\lambda_\pm^2\lambda_\mp^2j_{\pm}\right)(t,\cdot)\right\|_{L^\infty_x}&\lesssim \left\|\left(\lambda_\pm^2\lambda_\mp^2j_{\pm}\right)(t,\cdot)\right\|_{H^2_x}=\sum_{|\beta|\leq 2}\left\|\nabla^{|\beta|}(\lambda_\pm^2\lambda_\mp^2j_{\pm})(t,\cdot)\right\|_{L^2_x}\\
	&\lesssim\sum_{|\beta|\leq 2}\sum_{|\gamma|\leq|\beta|\atop|\gamma|\neq 0}\left\|\left(\nabla^{|\gamma|}\left(\lambda_\pm^2\lambda_\mp^2\right)j^{[|\beta|-|\gamma|]}_{\pm}\right)(t,\cdot)\right\|_{L^2_x}+\sum_{|\beta|\leq 2}\left\|\left(\lambda_\pm^2\lambda_\mp^2j_{\pm}^{[|\beta|]}\right)(t,\cdot)\right\|_{L^2_x}\\
	&\lesssim\sum_{|\beta|\leq 2}\sum_{|\gamma|\leq|\beta|\atop|\gamma|\neq 0}\left\| \left(\frac{1}{R}\lambda_\pm^2\lambda_\mp^2j^{[|\beta|-|\gamma|]}_{\pm}\right)(t,\cdot)\right\|_{L^2_x}+\sum_{|\beta|\leq 2}\left\|\left(\lambda_\pm^2\lambda_\mp^2j_{\pm}^{[|\beta|]}\right)(t,\cdot)\right\|_{L^2_x}\\
	&\lesssim\left(\frac{1}{R}+1\right)\mathcal{E}(t)^{1/2}.
\end{align*}
Based on this estimate and  \eqref{eq:z to j inequality},  the rest two inequalities can be proved similarly. Here we only retain the proof for the inequality on $\nabla z_\pm$:
\begin{align*}
\left\|\left(\lambda_\pm^2\lambda_\mp^2\nabla z_{\pm}\right)(t,\cdot)\right\|_{L^\infty_x}&\lesssim \left\|\left(\lambda_\pm^2\lambda_\mp^2\nabla z_{\pm}\right)(t,\cdot)\right\|_{H^2_x}=\sum_{|\beta|\leq 2}\left\|\nabla^{|\beta|}(\lambda_\pm^2\lambda_\mp^2\nabla z_{\pm})(t,\cdot)\right\|_{L^2_x}\\
&\lesssim\sum_{|\beta|\leq 2}\sum_{|\gamma|\leq|\beta|\atop|\gamma|\neq 0}\left\| \left(\nabla^{|\gamma|}\left(\lambda_\pm^2\lambda_\mp^2\right)\nabla z^{[|\beta|-|\gamma|]}_{\pm}\right)(t,\cdot)\right\|_{L^2_x}+\sum_{|\beta|\leq 2}\left\|\left(\lambda_\pm^2\lambda_\mp^2\nabla z_{\pm}^{[|\beta|]}\right)(t,\cdot)\right\|_{L^2_x}\\
&\lesssim \sum_{|\beta|\leq 2}\sum_{|\gamma|\leq|\beta|\atop|\gamma|\neq 0}\left\| \left(\frac{1}{R}\lambda_\pm^2\lambda_\mp^2\nabla z^{[|\beta|-|\gamma|]}_{\pm}\right)(t,\cdot)\right\|_{L^2_x}+\sum_{|\beta|\leq 2}\left\|\left(\lambda_\pm^2\lambda_\mp^2\nabla z_{\pm}^{[|\beta|]}\right)(t,\cdot)\right\|_{L^2_x}\\
&\lesssim\sum_{|\beta|\leq 2}\sum_{|\gamma|\leq|\beta|\atop|\gamma|\neq 0}\frac{1}{R}\Big(\left\|\left(\lambda_\pm^2\lambda_\mp^2z_{\pm}\right)(t,\cdot)\right\|_{L^2_x}+\sum_{l=0}^{|\beta|-|\gamma|}\left\|\left(\lambda_\pm^2\lambda_\mp^2j^{[l]}_{\pm}\right)(t,\cdot)\right\|_{L^2_x}\Big)\\
&\ \ \ \ +\sum_{|\beta|\leq 2}\Big(\left\|\left(\lambda_\pm^2\lambda_\mp^2z_{\pm}\right)(t,\cdot)\right\|_{L^2_x}+\sum_{l=0}^{|\beta|}\left\|\left(\lambda_\pm^2\lambda_\mp^2j^{[l]}_{\pm}\right)(t,\cdot)\right\|_{L^2_x}\Big)\\
&\lesssim\left(\frac{1}{R}+1\right)\mathcal{E}(t)^{1/2}.
\end{align*}
Thus we have proved the lemma.
\end{proof}

\begin{lemma}[Weighted spacetime estimates by using flux]\label{lemma:special}
		In $\mathbb{R}^3$,	for all $k\leq 2$, we have 
	\begin{align*}
		&\left\|\lambda_\pm\lambda_\mp z_\pm\right\|_{L^2_{[0,t]} L^2_x}\lesssim  \mathcal{E}(t)^{1/2},\\
		&\left\|\lambda_\pm\lambda_\mp j^{[k]}_\pm\right\|_{L^2_{[0,t]} L^2_x}\lesssim   \mathcal{E}(t)^{1/2},\\
	&\left\|\lambda_\pm\lambda_\mp\nabla z^{[k]}_\pm\right\|_{L^2_{[0,t]} L^2_x}\lesssim   \mathcal{E}(t)^{1/2}.
	\end{align*}
Here, we have used the following notation:
\[\left\|f\right\|_{L^2_{[0,t]}}:=\left(\int_0^t|f(\tau)|^2\, d\tau\right)^{1/2}.\]
\end{lemma}
 \begin{proof}
Using the Fubini's theorem and weight property \eqref{eq:integral}, 
 we can prove	the first two estimates  in the same manner:
 	\begin{align*}
 		\left\|\lambda_\pm\lambda_\mp z_\pm\right\|_{L^2_{[0,t]} L^2_x}
 		&=\left(\int_0^t\int_{\mathbb{R}^3}\frac{\lambda^4_\pm\lambda^4_\mp|z_\pm|^2}{\lambda^2_\pm\lambda^2_\mp}
 		dxd\tau\right)^{1/2}\\
 		&\lesssim\left(\int_{\mathbb{R}}\int_{C_{u_{\pm}}^{\pm,t}}\frac{\lambda^4_\pm\lambda^4_\mp| z_\pm|^2}{\lambda^2_\pm\lambda^2_\mp}
 		d\sigma_{\pm}du_\pm\right)^{1/2}\\
 		&\lesssim\left(\int_{\mathbb{R}}\frac{1}{\lambda^2_\pm\lambda^2_\mp}du_\pm\right)^{1/2}\left(\sup_{u_{\pm}\in\mathbb{R}}\int_{C_{u_{\pm}}^{\pm,t}}
 		\lambda^4_\pm\lambda^4_\mp| z_\pm|^2
 	d\sigma_{\pm}\right)^{1/2}\\
 		&\lesssim (F_{\pm,2}(t))^{1/2}\lesssim   \mathcal{E}(t)^{1/2},\\
 	\left\|\lambda_\pm\lambda_\mp j^{[k]}_\pm\right\|_{L^2_{[0,t]} L^2_x}
 	&=\left(\int_0^t\int_{\mathbb{R}^3}\frac{\lambda^4_\pm\lambda^4_\mp|j^{[k]}_\pm|^2}{\lambda^2_\pm\lambda^2_\mp}
 	dxd\tau\right)^{1/2}\\
 	&\lesssim\left(\int_{\mathbb{R}}\int_{C_{u_{\pm}}^{\pm,t}}\frac{\lambda^4_\pm\lambda^4_\mp|j^{[k]}_\pm|^2}{\lambda^2_\pm\lambda^2_\mp}
 	d\sigma_{\pm}du_\pm\right)^{1/2}\\
 	&\lesssim\left(\int_{\mathbb{R}}\frac{1}{\lambda^2_\pm\lambda^2_\mp}du_\pm\right)^{1/2}\left(\sup_{u_{\pm}\in\mathbb{R}}\int_{C_{u_{\pm}}^{\pm,t}}
	\lambda^4_\pm\lambda^4_\mp|j^{[k]}_\pm|^2
 d\sigma_{\pm}\right)^{1/2}\\
 	&\lesssim (F^{[k]}_{\pm,2}(t))^{1/2}\lesssim   \mathcal{E}(t)^{1/2}.
 \end{align*}
Based on these two estimates, we can further apply \eqref{eq:z to j inequality} to derive
\begin{align*}
\left\|\lambda_\pm\lambda_\mp\nabla z^{[k]}_\pm\right\|_{L^2_{[0,t]} L^2_x}
&\lesssim \left\|\lambda_\pm\lambda_\mp z_\pm\right\|_{L^2_{[0,t]} L^2_x}+\sum_{l=0}^{k}	\left\|\lambda_\pm\lambda_\mp j^{[l]}_\pm\right\|_{L^2_{[0,t]} L^2_x}\lesssim  \mathcal{E}(t)^{1/2}.
\end{align*}
The proof of the lemma is now complete.
 \end{proof}

\begin{lemma}[Weighted Sobolev inequality in spacetime]\label{lemma:flux-Linfty}
	In $\mathbb{R}^3$,	 we have 
	\begin{align*}
		&\left\|\lambda_\pm\lambda_\mp z_\pm\right\|_{L^2_{[0,t]} L^\infty_x}\lesssim \left\|\lambda_\pm\lambda_\mp z_\pm\right\|_{L^2_{[0,t]} H^2_x}\lesssim  \mathcal{E}(t)^{1/2},\\
		&\left\|\lambda_\pm\lambda_\mp\nabla z_\pm\right\|_{L^2_{[0,t]} L^\infty_x}\lesssim  \left\|\lambda_\pm\lambda_\mp\nabla z_\pm\right\|_{L^2_{[0,t]} H^2_x}\lesssim  \mathcal{E}(t)^{1/2}.
	\end{align*}
\end{lemma}

\begin{proof}
	These two inequalities can be proved in the same manner by using the standard Sobolev inequality on $\mathbb{R}^3$, weight property 
	\eqref{differentiate weights coro},  div-curl Lemma \ref{lemma:divcurl}, and the weighted spacetime estimates in Lemma \ref{lemma:special}. 
 Here we only give the details of the second inequality:
\begin{align*}
	\left\|\lambda_\pm\lambda_\mp\nabla z_{\pm}\right\|_{L^2_{[0,t]}L^\infty_x}&\lesssim \left\|\lambda_\pm\lambda_\mp\nabla z_{\pm}\right\|_{L^2_{[0,t]}H^2_x}=\sum_{|\beta|\leq 2}\left\|\nabla^{|\beta|}(\lambda_\pm\lambda_\mp\nabla z_{\pm})\right\|_{L^2_{[0,t]}L^2_x}\\
	&\lesssim\sum_{|\beta|\leq 2}\sum_{|\gamma|\leq|\beta|\atop|\gamma|\neq 0}\left\| \nabla^{|\gamma|}\left(\lambda_\pm\lambda_\mp\right)\nabla z^{[|\beta|-|\gamma|]}_{\pm}\right\|_{L^2_{[0,t]}L^2_x}+\sum_{|\beta|\leq 2}\left\|\lambda_\pm\lambda_\mp\nabla z_{\pm}^{[|\beta|]}\right\|_{L^2_{[0,t]}L^2_x}\\
	&\lesssim \sum_{|\beta|\leq 2}\sum_{|\gamma|\leq|\beta|\atop|\gamma|\neq 0}\left\| \frac{1}{R}\lambda_\pm\lambda_\mp\nabla z^{[|\beta|-|\gamma|]}_{\pm}\right\|_{L^2_{[0,t]}L^2_x}+\sum_{|\beta|\leq 2}\left\|\lambda_\pm\lambda_\mp\nabla z_{\pm}^{[|\beta|]}\right\|_{L^2_{[0,t]}L^2_x}\\
	&\lesssim\sum_{|\beta|\leq 2}\sum_{|\gamma|\leq|\beta|\atop|\gamma|\neq 0}\frac{1}{R}\Big(\left\|\lambda_\pm\lambda_\mp z_{\pm} \right\|_{L^2_{[0,t]}L^2_x}+\sum_{l=0}^{|\beta|-|\gamma|}\left\|\lambda_\pm\lambda_\mp j^{[l]}_{\pm}\right\|_{L^2_{[0,t]}L^2_x}\Big)\\
	&\ \ \ \ +\sum_{|\beta|\leq 2}\Big(\left\|\lambda_\pm\lambda_\mp z_{\pm}\right\|_{L^2_{[0,t]}L^2_x}+\sum_{l=0}^{|\beta|}\left\|\lambda_\pm\lambda_\mp j^{[l]}_{\pm}\right\|_{L^2_{[0,t]}L^2_x}\Big)\\
	&\lesssim\left(\frac{1}{R}+1\right)\mathcal{E}(t)^{1/2}.
\end{align*}
This proves the lemma. 
\end{proof}

 \section{Two types of linear energy estimates}\label{sec:linear}

Let $f_\pm$ and $\rho_\pm$ be smooth vector fields defined on $\mathbb{R}\times \mathbb{R}^3$ with sufficiently fast decay in $x$-variables. 
In this section, we derive two types of 
weighted energy estimates 
for the following linear system:
\begin{equation}\label{eq:linear eq}\begin{cases}
		& L_- f_{+}  =\rho_{+},\\
		& L_+ f_{-}  =\rho_{-},\\
		&\operatorname{div} f_+=0,\\
		&\operatorname{div} f_-=0.
\end{cases}\end{equation}
One is $L^2$-type, and the other is $L^{3/2}$-type.  
In each type, we  always match unified weights $\lambda^2_{\pm}\lambda^2_{\mp}$ for any vector field $f_\pm$.

\begin{lemma}[$L^2$-type  linear energy estimates]\label{lemma:linear EE}
	For all  $t\in [0,t^*]$, we have the following $L^2$-type energy estimates for the linear system \eqref{eq:linear eq}:
	\begin{align*}
		&\ \ \ \ \sum_{+,-}	\int_{\Sigma_{t}}\lambda^4_{\pm}\lambda^4_{\mp}|f_{\pm}|^2 dx +\sum_{+,-}\sup_{u_{\pm}\in\mathbb{R}}\int_{C_{u_{\pm}}^{\pm,t}}\lambda^4_{\pm}\lambda^4_{\mp}|f_{\pm}|^2 d\sigma_{\pm}\\
		&\lesssim \sum_{+,-}\int_{\Sigma_{0}}\lambda^4_{\pm}\lambda^4_{\mp}|f_{\pm}|^2 dx+\sum_{+,-}\int_{0}^{t}\int_{\Sigma_{\tau}}\lambda^4_{\pm}\lambda^4_{\mp}\big|f_{\pm}\big|\big|\rho_{\pm}\big|dxd\tau +\sum_{+,-}\frac{1}{R}\int_{0}^{t}\int_{\Sigma_{\tau}} \lambda^4_{\pm}\lambda^4_{\mp}\big|f_{\pm}\big|^2dxd\tau.
		\stepcounter{equation}\tag{\theequation}\label{eq:linear EE}
	\end{align*}
\end{lemma}

\begin{proof}
We need to use the following direct consequence of  \eqref{eq:linear eq}:
\[ L_\mp \left|f_{\pm} \right|^{2}  =2\rho_{\pm} \left|f_{\pm} \right|\operatorname{sgn} \left(f_{\pm}\right).\]
By \eqref{eq:lambda-condition'}, 
it is clear that   $$L_\mp\left(\lambda_\pm^4\lambda_\mp^4\right)
=4\lambda_\pm^3\lambda_\mp^4L_\mp\lambda_\pm+4\lambda_\pm^4\lambda_\mp^3L_\mp\lambda_\mp\lesssim \frac{1}{R}\lambda_\pm^4\lambda_\mp^4.$$
Hence we obtain
\begin{equation}\label{eq:integrate}
	L_\mp \left(\lambda_\pm^4\lambda_\mp^4 \left|f_{\pm} \right|^{2}\right)  \lesssim\lambda_\pm^4\lambda_\mp^4|\rho_{\pm}| \left|f_{\pm} \right| +\frac{1}{R}\lambda_\pm^4\lambda_\mp^4 \left|f_{\pm} \right|^{2}.
\end{equation}

In what follows, for real numbers $t$, $u_+^1$, $u_+^2$, $u_-^1$ and $u_-^2$ given with $u_+^1<u_+^2$ and $u_-^1<u_-^2$,  
we define the following hypersurfaces:
\[\Sigma_\tau^{\left[u_\pm^1,u_\pm^2\right]}:=\bigcup_{u_\pm\in\big[u_\pm^1,u_\pm^2\big]}\big(C_{u_\pm}^{\pm,t}\cap\Sigma_{t}\big).\]

We first integrate the inequality  \eqref{eq:integrate} over $[0,t]\times \Sigma_\tau^{\left[u_\pm,+\infty\right]}$ to get
\begin{align*}
\int_0^t \int_{\Sigma_\tau^{\left[u_\pm,+\infty\right]}}L_\mp \left(\lambda^4_{\pm}\lambda^4_{\mp}\left|f_{\pm} \right|^{2}\right) dxd\tau \lesssim\int_0^t \int_{\Sigma_\tau^{\left[u_\pm,+\infty\right]}}\lambda^4_{\pm}\lambda^4_{\mp}|\rho_{\pm}| \left|f_{\pm} \right|  dxd\tau+\frac{1}{R}\int_0^t \int_{\Sigma_\tau^{\left[u_\pm,+\infty\right]}} \lambda^4_{\pm}\lambda^4_{\mp}\left|f_{\pm} \right|^{2} dxd\tau.
\end{align*}
Let $\widetilde{\operatorname{div}}$ denote the divergence of $\mathbb{R}^4$ with respect to the standard Euclidean metric. 
Since $\operatorname{div}Z_\pm=0$, we have $\widetilde{\operatorname{div}}L_\pm=0$. 
By virtue of the Stokes formula, the left hand side above can be rewritten as
\begin{align*}
&\ \ \ \ \int_0^t \int_{\Sigma_\tau^{\left[u_\pm,+\infty\right]}} \widetilde{\operatorname{div}}\left(\lambda^4_{\pm}\lambda^4_{\mp}\left|f_{\pm} \right|^{2}L_\mp\right) dxd\tau\\
&=\int_{\Sigma_t^{\left[u_\pm,+\infty\right]}} \lambda^4_{\pm}\lambda^4_{\mp}\left|f_{\pm} \right|^{2}\langle L_\mp,T\rangle dx -\int_{\Sigma_0^{\left[u_\pm,+\infty\right]}} \lambda^4_{\pm}\lambda^4_{\mp}\left|f_{\pm} \right|^{2}\langle L_\mp,T\rangle dx +\sum_{k=u_\pm,+\infty}\int_{C^\pm_{k}} \lambda^4_{\pm}\lambda^4_{\mp}\left|f_{\pm} \right|^{2}\langle L_\mp,\nu_k^\pm \rangle d\sigma_\pm\\
&\sim\int_{\Sigma_t^{\left[u_\pm,+\infty\right]}} \lambda^4_{\pm}\lambda^4_{\mp}\left|f_{\pm} \right|^{2}  dx -\int_{\Sigma_0^{\left[u_\pm,+\infty\right]}} \lambda^4_{\pm}\lambda^4_{\mp}\left|f_{\pm} \right|^{2}  dx+\int_{C^\pm_{u_\pm}} \lambda^4_{\pm}\lambda^4_{\mp}\left|f_{\pm} \right|^{2}  d\sigma_\pm,
\end{align*}
where we have used the facts $\langle L_\mp,T \rangle=1$ and $\langle L_\mp,\nu_k^\pm \rangle\sim 1$  in the last line. Then we infer that
\begin{align*}
&\ \ \ \ \sum_{+,-}\int_{\Sigma_t^{\left[u_\pm,+\infty\right]}} \lambda^4_{\pm}\lambda^4_{\mp}\left|f_{\pm} \right|^{2}  dx+\sum_{+,-}\int_{C^\pm_{u_\pm}} \lambda^4_{\pm}\lambda^4_{\mp}\left|f_{\pm} \right|^{2}  d\sigma_\pm\\ &\lesssim\sum_{+,-}\int_{\Sigma_0^{\left[u_\pm,+\infty\right]}} \lambda^4_{\pm}\lambda^4_{\mp}\left|f_{\pm} \right|^{2}  dx+\sum_{+,-}\int_0^t \int_{\Sigma_\tau^{\left[u_\pm,+\infty\right]}}\lambda^4_{\pm}\lambda^4_{\mp}|\rho_{\pm}| \left|f_{\pm} \right| dxd\tau\\
&\ \ \ \ +\sum_{+,-}\frac{1}{R}\int_0^t \int_{\Sigma_\tau^{\left[u_\pm,+\infty\right]}}\lambda^4_{\pm}\lambda^4_{\mp} \left|f_{\pm} \right|^2 dxd\tau.\stepcounter{equation}\tag{\theequation}\label{eq1}
\end{align*}

Next, we  integrate the inequality  \eqref{eq:integrate} over $[0,t]\times \Sigma_\tau^{\left[-\infty,+\infty\right]}$ to derive
	\begin{align*}
	&\ \ \ \ 	\int_0^t \int_{\Sigma_\tau^{\left[-\infty,+\infty\right]}}L_\mp \left(\lambda^4_{\pm}\lambda^4_{\mp}\left|f_{\pm} \right|^{2}\right) dxd\tau\\ &\lesssim\int_0^t \int_{\Sigma_\tau^{\left[-\infty,+\infty\right]}}\lambda^4_{\pm}\lambda^4_{\mp}|\rho_{\pm}| \left|f_{\pm} \right|  dxd\tau+\frac{1}{R}\int_0^t \int_{\Sigma_\tau^{\left[-\infty,+\infty\right]}} \lambda^4_{\pm}\lambda^4_{\mp} \left|f_{\pm} \right|^2 dxd\tau.
	\end{align*}
Similarly, 	we can use  $\widetilde{\operatorname{div}}L_\mp=0$ and the Stokes formula to rephrase the left hand side  as 
	\begin{align*}
		&\ \ \ \ \int_0^t \int_{\Sigma_\tau^{\left[-\infty,+\infty\right]}} \widetilde{\operatorname{div}}\left(\lambda^4_{\pm}\lambda^4_{\mp}\left|f_{\pm} \right|^{2}L_\mp\right) dxd\tau\\
		&=\int_{\Sigma_t^{\left[-\infty,+\infty\right]}} \lambda^4_{\pm}\lambda^4_{\mp}\left|f_{\pm} \right|^{2}\langle L_\mp,T\rangle dx -\int_{\Sigma_0^{\left[-\infty,+\infty\right]}} \lambda^4_{\pm}\lambda^4_{\mp}\left|f_{\pm} \right|^{2}\langle L_\mp,T\rangle dx\\ 
		&\ \ \ \ +\sum_{k=+\infty,-\infty}\int_{C^\mp_{k}} \lambda^4_{\pm}\lambda^4_{\mp}\left|f_{\pm} \right|^{2}\langle L_\mp,\nu_k^\mp \rangle d\sigma_\mp\\
		&=\int_{\Sigma_t^{\left[-\infty,+\infty\right]}} \lambda^4_{\pm}\lambda^4_{\mp}\left|f_{\pm} \right|^{2}  dx -\int_{\Sigma_0^{\left[-\infty,+\infty\right]}} \lambda^4_{\pm}\lambda^4_{\mp}\left|f_{\pm} \right|^{2}  dx,
	\end{align*}
	where we have used the facts $\langle L_\mp,T \rangle=1$
	and $\langle L_\mp,\nu_k^\mp \rangle=0$ in the last line. Therefore we obtain 
	\begin{align*}
 \sum_{+,-}\int_{\Sigma_t^{\left[-\infty,+\infty\right]}} \lambda^4_{\pm}\lambda^4_{\mp}\left|f_{\pm} \right|^{2}  dx &\lesssim\sum_{+,-}\int_{\Sigma_0^{\left[-\infty,+\infty\right]}} \lambda^4_{\pm}\lambda^4_{\mp}\left|f_{\pm} \right|^{2}  dx+\sum_{+,-}\int_0^t \int_{\Sigma_\tau^{\left[-\infty,+\infty\right]}}\lambda^4_{\pm}\lambda^4_{\mp}|\rho_{\pm}| \left|f_{\pm} \right|  dxd\tau\\
		&\ \ \ \ +\sum_{+,-}\frac{1}{R}\int_0^t \int_{\Sigma_\tau^{\left[-\infty,+\infty\right]}}\lambda^4_{\pm}\lambda^4_{\mp} \left|f_{\pm} \right|^2 dxd\tau.\stepcounter{equation}\tag{\theequation}\label{eq2}
\end{align*}

Putting \eqref{eq1} and \eqref{eq2} together finally yields \eqref{eq:linear EE}. 
\end{proof}

\begin{lemma}[$L^{3/2}$-type linear  energy estimates]\label{lemma:linear EE2}
	For all $t\in [0,t^*]$, we have the following $L^{3/2}$-type energy estimates for the linear system \eqref{eq:linear eq}:
 \begin{align*} 
			&\ \ \ \ 	\sum_{+,-}	\int_{\Sigma_{t}}\lambda^3_{\pm}\lambda^3_{\mp}|f_{\pm}|^{3/2} dx+\sum_{+,-}\sup_{u_{\pm}\in\mathbb{R}}\int_{C_{u_{\pm}}^{\pm,t}}\lambda^3_{\pm}\lambda^3_{\mp}|f_{\pm}|^{3/2} d\sigma_{\pm}\\
			& \lesssim \sum_{+,-}\int_{\Sigma_{0}}\lambda^3_{\pm}\lambda^3_{\mp}|f_{\pm}|^{3/2} dx
			+\sum_{+,-}\int_{0}^{t}\int_{\Sigma_{\tau}}\lambda^3_{\pm}\lambda^3_{\mp}\big|f_{\pm}\big|^{1/2}\big|\rho_{\pm}\big|dxd\tau + \sum_{+,-}\frac{1}{R}\int_{0}^{t}\int_{\Sigma_{\tau}} \lambda^3_{\pm}\lambda^3_{\mp}\big|f_{\pm}\big|^{3/2}dxd\tau.\stepcounter{equation}\tag{\theequation}\label{eq:linear EE-2}
		\end{align*}
\end{lemma} 
\begin{proof}
The derivation of  the $L^{3/2}$ version can be conducted in   a similar manner.
Based on the facts
\[
L_\mp  \left|f_{\pm} \right|^{3/2}  =\frac{3}{2}\rho_{\pm} \left|f_{\pm} \right|^{1/2} \operatorname{sgn} \left(f_{\pm}\right)\]
and
$$L_\mp\left(\lambda_\pm^3\lambda_\mp^3\right)
=3\lambda_\pm^2\lambda_\mp^3L_\mp\lambda_\pm+3\lambda_\pm^3\lambda_\mp^2L_\mp\lambda_\mp\lesssim\frac{1}{R} \lambda_\pm^3\lambda_\mp^3,$$
we have
\begin{equation}\label{eq:integrate2}
	L_\mp (\lambda_\pm^3\lambda_\mp^3 \left|f_{\pm} \right|^{3/2})  \lesssim\lambda_\pm^3\lambda_\mp^3|\rho_{\pm}| \left|f_{\pm} \right|^{1/2} +\frac{1}{R}\lambda_\pm^3\lambda_\mp^3 \left|f_{\pm} \right|^{3/2}.
\end{equation}
Integrating the inequality  \eqref{eq:integrate2} over $[0,t]\times \Sigma_\tau^{\left[u_\pm,+\infty\right]}$ and using the Stokes formula can lead us to 
	\begin{align*}
		&\ \ \ \ \sum_{+,-}\int_{\Sigma_t^{\left[u_\pm,+\infty\right]}} \lambda^3_{\pm}\lambda^3_{\mp}\left|f_{\pm} \right|^{3/2} dx+\sum_{+,-}\int_{C^\pm_{u_\pm}} \lambda^3_{\pm}\lambda^3_{\mp}\left|f_{\pm} \right|^{3/2} d\sigma_\pm\\ &\lesssim\sum_{+,-}\int_{\Sigma_0^{\left[u_\pm,+\infty\right]}} \lambda^3_{\pm}\lambda^3_{\mp}\left|f_{\pm} \right|^{3/2} dx+ \sum_{+,-}\int_0^t \int_{\Sigma_\tau^{\left[u_\pm,+\infty\right]}}\lambda^3_{\pm}\lambda^3_{\mp}|\rho_{\pm}| \left|f_{\pm} \right|^{1/2} dxd\tau\\
		&\ \ \ \ + \sum_{+,-}\frac{1}{R}\int_0^t \int_{\Sigma_\tau^{\left[u_\pm,+\infty\right]}}\lambda^3_{\pm}\lambda^3_{\mp}\left|f_{\pm} \right|^{3/2} dxd\tau.
\end{align*}
On the other hand, we integrate the inequality  \eqref{eq:integrate2} over $[0,t]\times \Sigma_\tau^{\left[-\infty,+\infty\right]}$ and use the Stokes formula to deduce
	\begin{align*}
		&\ \ \ \ \sum_{+,-}\int_{\Sigma_t^{\left[-\infty,+\infty\right]}} \lambda^3_{\pm}\lambda^3_{\mp}\left|f_{\pm} \right|^{3/2} dx\\ &\lesssim\sum_{+,-}\int_{\Sigma_0^{\left[-\infty,+\infty\right]}} \lambda^3_{\pm}\lambda^3_{\mp}\left|f_{\pm} \right|^{3/2}  dx+ \sum_{+,-}\int_0^t \int_{\Sigma_\tau^{\left[-\infty,+\infty\right]}}\lambda^3_{\pm}\lambda^3_{\mp}|\rho_{\pm}|  \left|f_{\pm} \right|^{1/2}dxd\tau\\ 
		&\ \ \ \ +\sum_{+,-}\frac{1}{R}\int_0^t \int_{\Sigma_\tau^{\left[-\infty,+\infty\right]}}\lambda^3_{\pm}\lambda^3_{\mp}  \left|f_{\pm} \right|^{3/2}  dxd\tau.
\end{align*}
Adding up the above two inequalities finally gives rise to the desired result.
\end{proof}

\begin{remark} 	 
	The nonlinear terms in the rest of proof will be  controlled by combining the weighted $L^2$ and 
	$L^{3/2}$ estimate together, and the transport structure must be
	combined with carefully designed unified weights in order to
	recover the missing integrability.
	This difficulty is also the reason why the
	$L^{3/2}$ component is treated separately throughout the proof.
\end{remark}

\begin{remark}
	The role of the unified weights is not merely technical.
	Their construction allows the weighted $L^2$ and weighted
	$L^{3/2}$ energies to be propagated within a single bootstrap
	framework while preserving the transport structure along both
	characteristic directions.
	This compatibility makes it unnecessary to introduce separate
	weight systems for different energy levels, and hopefully will be useful
	in other problems involving multiple weights and energies.
	\end{remark}

\section{Proof of 
 Theorem \ref{thm}}\label{sec:energy}

The core of the proof of Theorem \ref{thm} is to demonstrate our improved energy estimates, which will be divided into four subsections. The lifespan estimate will directly follow from our improved energy estimates in the final subsection.

\subsection{$L^2$-type zero-order energy estimates}
In this subsection, we will apply the $L^2$-type linear energy estimates to the system \eqref{eq:MHD}, i.e. we take 
\[f_{\pm}=z_{\pm},\ \ \ \  \rho_{\pm}=-\nabla p
\]
in \eqref{eq:linear EE}. Hence we obtain 
\begin{equation*}
	\sum_{+,-}E_{\pm,2}(t)+\sum_{+,-}F_{\pm,2}(t)\lesssim \sum_{+,-}E_{\pm,2}(0)+\sum_{+,-}\underbrace{\int_{0}^{t}\int_{\Sigma_{\tau}}\lambda^4_{\pm}\lambda^4_{\mp}|z_{\pm}| |\nabla p|dxd\tau}_{\mathbf{I}}+\frac{1}{R}\int_{0}^{t}\sum_{+,-}E_{\pm,2}(\tau)d\tau,
\end{equation*}
which gives
\begin{equation*} 
	\sum_{+,-}E_{\pm,2}(t)+\sum_{+,-}F_{\pm,2}(t)\lesssim \sum_{+,-}E_{\pm,2}(0)+\sum_{+,-}\mathbf{I}+\frac{1}{R}\int_{0}^{t}\Big(\sum_{+,-}E_{\pm,2}(\tau)+\sum_{+,-}F_{\pm,2}(\tau)\Big)d\tau.
\end{equation*}

Let us bound  $\mathbf{I}$
in the rest of this subsection:
\begin{align*}
	\mathbf{I} &\lesssim\int_{0}^{t}\int_{\Sigma_{\tau}}\lambda_\pm^4\lambda_\mp^4|z_{\pm}| |\nabla p|dxd\tau\\
	&\lesssim\underbrace{\left\|\lambda_\pm\lambda_\mp z_{\pm}\right\|_{L^2_{[0,t]} L^2_x}}_{\text{by Lemma \ref{lemma:special}}} 
\Big(\underbrace{\int_{0}^{t}\int_{\Sigma_{\tau}}\lambda_\pm^6\lambda_\mp^6 |\nabla p|^2dxd\tau}_{\mathbf{I'}}\Big)^{1/2}\\
	&\lesssim   \mathcal{E}(t)^{1/2}\  \mathbf{I'}^{1/2}.
\end{align*}
To bound $\mathbf{I'}$, we need to express $\nabla p$ first. 
Since $\operatorname{div} z_{\pm}=0$,  taking divergence of the first equation in \eqref{eq:MHD} gives 
\begin{equation*}
	-\Delta p=\partial_{i}z^{j}_{\mp}\partial_{j}z^{i}_{\pm}.
\end{equation*}
Let $\theta(r)$ be a smooth cut-off function  so that $\theta(r)=1$ for $r\leq 1$ and $\theta(r)=0$ for $r\geq 2$.	
Based on the Newtonian potential and integration by parts, we have the following  pressure decomposition:
\begin{align*}  
		\nabla p(\tau,x)&=
		\frac{1}{4\pi}\int_{\mathbb{R}^3}\nabla_x\frac{1}{|x-x'|}\cdot\theta(|x-x'|)\cdot\big(\partial_{i}z^{j}_{\mp}\partial_{j}z^{i}_{\pm}\big)(\tau,x')dx'\\
		&\ \ +\frac{1}{4\pi}\int_{\mathbb{R}^3}\partial_{j}\partial_{i}\Big(\nabla_x\frac{1}{|x-x'|} \cdot\big(1-\theta(|x-x'|)\big)\Big)\cdot\big(z^{j}_{\mp}z^{i}_{\pm}\big)(\tau,x')dx'.
\end{align*}  
This implies
 \begin{equation*}
 	|\nabla p(\tau,x)|\leq\underbrace{\int_{|x-x'|\leq 2}\frac{|(\nabla  z_{\mp}\cdot\nabla  z_{\pm})(\tau,x')|}{|x-x'|^2}dx'}_{\mathbf{A}}+\underbrace{\int_{|x-x'|\geq 1}\frac{|( z_{\mp}\cdot  z_{\pm})(\tau,x')|}{|x-x'|^4}dx'}_{\mathbf{B}},
 \end{equation*}
and hence
\begin{equation*}
	\mathbf{I'} \lesssim 
	\underbrace{\int_{0}^{t}\int_{\Sigma_{\tau}}\lambda_\pm^6\lambda_\mp^6|\mathbf{A}|^2dxd\tau}_{\mathbf{I_A}}
	+\underbrace{\int_{0}^{t}\int_{\Sigma_{\tau}}\lambda_\pm^6\lambda_\mp^6 |\mathbf{B}|^2dxd\tau}_{\mathbf{I_B}}.
\end{equation*}

	For $\mathbf{I}_{\mathbf{A}}$, according to  weight properties in Lemma \ref{lemma:weights}, Young's inequality, H\"older inequality, Sobolev inequality in Lemma \ref{lemma:Sobolev}, and Lemma \ref{lemma:special}, we derive 
	\begin{align*}
		\mathbf{I}_{\mathbf{A}}
		&\lesssim 
		\int_{0}^{t}\Big\|\lambda^3_\pm(\tau,x)\lambda^3_\mp(\tau,x)\int_{|x-x'|\leq 2}\frac{|(\nabla  z_{\mp}\cdot\nabla  z_{\pm})(\tau,x')|}{|x-x'|^2}dx'\Big\|_{L^2_x}^2d\tau\\
		&
		\lesssim \int_{0}^{t}\Big\|\int_{|x-x'|\leq 2} \frac{\big(\lambda^3_\pm\lambda^3_\mp\big)(\tau,x')|(\nabla  z_{\mp}\cdot\nabla  z_{\pm})(\tau,x')|}{|x-x'|^2}dx'\Big\|_{L^2_x}^2d\tau\\
		&
		\lesssim
		\int_{0}^{t}\Big\|\frac{1}{|x|^2}\text{\Large$\chi$}_{|x|\leq 2}\Big\|_{L^1_x}^2 \big\|\lambda^2_\pm\lambda^2_\mp\nabla z_{\mp}\big\|_{L^\infty_x}^2\Big\|\lambda_\pm\lambda_\mp\nabla 
		z_{\pm}\Big\|_{L^2_x}^2 d\tau\\
		&\lesssim \max_{\tau\in[0,t]}\left\|\left(\lambda_\mp^2\lambda_\pm^2\nabla z_{\mp}\right)(\tau,\cdot)\right\|_{H^2_x}^2\cdot \left\|\lambda_\pm\lambda_\mp\nabla z_{\pm}\right\|^2_{L^2_{[0,t]} L^2_x}\\
		&\lesssim   \max_{\tau\in[0,t]} \mathcal{E}(\tau)^2.\stepcounter{equation}\tag{\theequation}\label{eq:IAl}
	\end{align*}
	Here, Young's inequality is balanced 
 on the fact that $\dfrac{1}{|x|^{2}}\text{\Large$\chi$}_{|x|\leq 2}\in L^1_x(\mathbb{R}^3)$
and  $\dfrac{1}{|x|^{2}}\text{\Large$\chi$}_{|x|\leq 2}\notin L^2_x(\mathbb{R}^3)$; and then	the use of Sobolev inequality in Lemma \ref{lemma:Sobolev}
necessitates that the closure of our energy estimates be at least third order. Specifically, the application of 
\[\big\|\lambda^2_\mp\lambda^2_\pm\nabla z_{\mp}\big\|_{L^\infty_x}\lesssim \left\|\lambda_\mp^2\lambda_\pm^2\nabla z_{\mp}\right\|_{H^2_x}\]
requires the third order derivatives of $z_\mp$,  making it difficult to expect that our total energy estimates can be improved  to 
below third order.

	For $\mathbf{I}_{\mathbf{B}}$, using weight properties, we obtain
	\begin{align*}
		\mathbf{I}_{\mathbf{B}}
		&\lesssim
		\int_{0}^{t}\Big\|\lambda^3_\pm(\tau,x)\lambda^3_\mp(\tau,x)\int_{|x-x'|\geq 1}\frac{|(z_{\mp}\cdot  z_{\pm})(\tau,x')|}{|x-x'|^4}dx'\Big\|_{L^2_x}^2d\tau\\
		&\lesssim \int_{0}^{t}\Big\|\int_{|x-x'|\geq 1}\Big(\big(\lambda^3_\pm\lambda^3_\mp\big)(\tau,x')+|x-x'|^{2(1+\delta)}\Big)\frac{|( z_{\mp}\cdot   z_{\pm})(\tau,x')|}{|x-x'|^4}dx'\Big\|_{L^2_x}^2d\tau\\
		&\lesssim
		\underbrace{\int_{0}^{t}\Big\|\int_{|x-x'|\geq 1}\big(\lambda^3_\pm\lambda^3_\mp\big)(\tau,x')\frac{|(z_{\mp}\cdot  z_{\pm})(\tau,x')|}{|x-x'|^4}dx'\Big\|_{L^2_x}^2d\tau}_{\mathbf{I}_{\mathbf{B}}'}
		 +\underbrace{\int_{0}^{t}\Big\|\int_{|x-x'|\geq 1}\frac{|(z_{\mp}\cdot  z_{\pm})(\tau,x')|}{|x-x'|^{4-2(1+\delta)}}dx'\Big\|_{L^2_x}^2d\tau}_{\mathbf{I}_{\mathbf{B}}''}.
	\end{align*}
Similar to \eqref{eq:IAl}, we can bound $\mathbf{I}_{\mathbf{B}}'$  as follows:
	\begin{align*}
		\mathbf{I}_{\mathbf{B}}'&
		\lesssim
		\int_{0}^{t}\Big\|\frac{1}{|x|^4}\text{\Large$\chi$}_{|x|\geq 1}\Big\|_{L^1_x}^2 \big\|\lambda^2_\pm\lambda^2_\mp z_{\mp}\big\|_{L^\infty_x}^2\Big\|\lambda_\pm\lambda_\mp
		z_{\pm}\Big\|_{L^2_x}^2 d\tau\\
		&\lesssim \max_{\tau\in[0,t]}\left\|\left(\lambda_\mp^2\lambda_\pm^2 z_{\mp}\right)(\tau,\cdot)\right\|_{H^2_x}^2\cdot\ \left\|\lambda_\pm\lambda_\mp z_{\pm}\right\|^2_{L^2_{[0,t]} L^2_x}\\
		&\lesssim   \max_{\tau\in[0,t]} \mathcal{E}(\tau)^2.
	\end{align*}
	It remains to bound $\mathbf{I}_{\mathbf{B}}''$.  	Notice that   $\dfrac{1}{|x|^{4-2(1+\delta)}}\text{\Large$\chi$}_{|x|\geq 1}\in L^2_x(\mathbb{R}^3)$ when 
\begin{equation}\label{def-delta}
\delta \in(0,\tfrac{1}{4})
\end{equation}
	 (this is the place where we need constraint on $\delta$). By weight properties, Young's inequality and H\"older inequality, we can derive
	\begin{align*}
		\mathbf{I}_{\mathbf{B}}''&
		\lesssim
		\int_{0}^{t}\Big\|\frac{1}{|x|^{4-2(1+\delta)}}\text{\Large$\chi$}_{|x|\geq 1}\Big\|_{L^2_x}^2\big\| z_{-}\cdot  z_{+}\big\|_{L^1_x}^2d\tau\\
		&
			\lesssim \int_{0}^{t}\big\|z_{-}\cdot z_{+}\big\|_{L^1_x}^2d\tau\\
		&
			\lesssim\int_{0}^{t}\frac{1}{\lambda_+^4\lambda_-^4}\Big\|\lambda_+^4\lambda_-^4 z_{-} z_{+}\Big\|_{L^1_x}^2d\tau\\
		&
		\lesssim \int_{0}^{t}\frac{1}{\lambda_+^4\lambda_-^4}\big\|\lambda_+^2\lambda_-^2 z_{-}\big\|_{L^2_x}^2\big\|\lambda_-^2\lambda_+^2 z_{+}\big\|_{L^2_x}^2d\tau\\
		&
		\lesssim   \max_{\tau\in[0,t]} \mathcal{E}(\tau)^2 \int_{0}^{t}\frac{1}{\lambda_+^4\lambda_-^4}d\tau\lesssim  \max_{\tau\in[0,t]} \mathcal{E}(\tau)^2. 
	\end{align*}
	Thus we conclude that 
	\begin{equation*}
		\mathbf{I}_{\mathbf{B}}\lesssim  \max_{\tau\in[0,t]} \mathcal{E}(\tau)^2. 
	\end{equation*}

Up to now, we can successively obtain
\[\mathbf{I'}\lesssim   \max_{\tau\in[0,t]} \mathcal{E}(\tau)^2,\]
then
\[\mathbf{I}\lesssim   \max_{\tau\in[0,t]} \mathcal{E}(\tau)^{3/2},\]
and hence
\begin{equation}\label{eq:LOW}
	\sum_{+,-}E_{\pm,2}(t)+\sum_{+,-}F_{\pm,2}(t)\lesssim \sum_{+,-}E_{\pm,2}(0)+  \max_{\tau\in[0,t]} \mathcal{E}(\tau)^{3/2}+\frac{1}{R}\int_{0}^{t}\Big(\sum_{+,-}E_{\pm,2}(\tau)+\sum_{+,-}F_{\pm,2}(\tau)\Big)d\tau.
\end{equation}

\subsection{$L^2$-type higher-order (first-order, second-order and third-order) energy estimates}
For higher-order derivatives to be taken on the system \eqref{eq:MHD}, we will  replace one order of  derivatives  with a curl.

To this aim, we consider the following vorticity system, given by   \eqref{eq:curlMHD-2}, which is obtained by taking curl on the system \eqref{eq:MHD}:
\begin{equation}\label{eq:curlMHD-2}\begin{cases}
		&\partial_{t}j_{+}+Z_-\cdot \nabla j_{+} =-\nabla z_{-}\wedge\nabla z_{+},\\
		&\partial_{t}j_{-}+Z_+\cdot \nabla j_{-} =-\nabla z_{+}\wedge\nabla z_{-},\\
		&\operatorname{div} j_{+}=0,\\
		&\operatorname{div} j_{-}=0.
\end{cases}\end{equation}
For a given multi-index $\beta$ with $0\leq|\beta|\leq 2$, we  commute $\partial^{\beta}$ derivatives with the higher-order system \eqref{eq:curlMHD-2} to obtain
\begin{equation}\begin{cases}\label{eq:DcurlMHD-2}
		\partial_t j^{(\beta)}_{+}+Z_-\cdot\nabla j^{(\beta)}_{+}&=\rho^{(\beta)}_{+},\\
		\partial_t j^{(\beta)}_{-}+Z_+\cdot\nabla j^{(\beta)}_{-}&=\rho^{(\beta)}_{-},\\
		\operatorname{div} j^{(\beta)}_{+}=0,\\
		\operatorname{div} j^{(\beta)}_{-}=0,
\end{cases}\end{equation}
where  
\begin{equation*} 
	\rho_\pm^{(\beta)}:=-\partial^\beta(\nabla z_\mp\wedge\nabla z_\pm)-[\partial^\beta,z_\mp\cdot\nabla] j_\pm
\end{equation*}
can be bounded by
\begin{equation*}
	\big|\rho_\pm^{(\beta)}\big|\leq\sum_{\gamma\leq\beta} C_\beta^\gamma\big|\nabla z_\mp^{(\gamma)}\big|\big|\nabla z_\pm^{(\beta-\gamma)}\big|
	+\sum_{0\neq\gamma\leq\beta} C_\beta^\gamma\big|z_\mp^{(\gamma)}\big|\big|\nabla j_\pm^{(\beta-\gamma)}\big|
	\lesssim\sum_{s\leq|\beta|} \big|\nabla z_\mp^{[s]}\big|\big|\nabla z_\pm^{[|\beta|-s]}\big|.
\end{equation*}

Now we apply the $L^2$-type linear energy estimates   to the system \eqref{eq:DcurlMHD-2}, i.e. we take 
\[f_{\pm}=j^{(\beta)}_{\pm},\ \ \ \  \rho_{\pm}=\rho_\pm^{(\beta)}
\]
in \eqref{eq:linear EE}. Thus we derive
\begin{align*}
	&\ \ \ \ \sum_{+,-}	E^{(\beta)}_{\pm,2}(t)+\sum_{+,-}F^{(\beta)}_{\pm,2}(t)\\
	&\lesssim \sum_{+,-}E^{(\beta)}_{\pm,2}(0) +\sum_{+,-}\int_{0}^{t}\int_{\Sigma_{\tau}}\lambda_\pm^4\lambda_\mp^4 \big|j^{(\beta)}_{\pm}\big|\big|\rho^{(\beta)}_{\pm}\big|dxd\tau +\frac{1}{R}\int_{0}^{t}\sum_{+,-}	E^{(\beta)}_{\pm,2}(\tau)d\tau\\
	&\lesssim \sum_{+,-}E^{(\beta)}_{\pm,2}(0) +\sum_{+,-}\sum_{s\leq|\beta|}\underbrace{\int_{0}^{t}\int_{\Sigma_{\tau}}\lambda_\pm^4\lambda_\mp^4 \big|j^{(\beta)}_{\pm}\big|\big|\nabla z_\mp^{[s]}\big|\big|\nabla z_\pm^{[|\beta|-s]}\big|dxd\tau}_{\mathbf{J}}+\frac{1}{R}\int_{0}^{t}\Big(\sum_{+,-}	E^{(\beta)}_{\pm,2}(\tau)+\sum_{+,-}F^{(\beta)}_{\pm,2}(\tau)\Big)d\tau. 
\end{align*}

To bound $\mathbf{J}$, we classify the terms in the integrand into three cases according to the value of $|\beta|$. The following estimates  are mainly based on  H\"older inequality, Sobolev inequality (Lemma \ref{lemma:Sobolev}, Lemma \ref{lemma:flux-Linfty}),    and the spacetime estimates (Lemma \ref{lemma:special}).

\textbf{Case 1}: $|\beta|=0$. We have to bound the term
\begin{align*}
	\mathbf{J} &\lesssim\int_{0}^{t}\int_{\Sigma_{\tau}}\lambda_\pm^4\lambda_\mp^4|j_{\pm}| |\nabla z_\mp||\nabla z_\pm|dxd\tau\\
	&\lesssim\left\|\lambda_\pm^2\lambda_\mp^2j_{\pm}\right\|_{L^\infty_\tau L^\infty_x}\left\|\lambda_\mp\lambda_\pm\nabla z_\mp\right\|_{L^2_{[0,t]} L^2_x}\left\|\lambda_\pm\lambda_\mp\nabla z_\pm\right\|_{L^2_{[0,t]} L^2_x}\\
	&\lesssim\left\|\lambda_\pm^2\lambda_\mp^2j_{\pm}\right\|_{L^\infty_\tau H^2_x}\left\|\lambda_\mp\lambda_\pm\nabla z_\mp\right\|_{L^2_{[0,t]} L^2_x}\left\|\lambda_\pm\lambda_\mp\nabla z_\pm\right\|_{L^2_{[0,t]} L^2_x}\\
	&\lesssim 
	  \max_{\tau\in[0,t]} \mathcal{E}(\tau)^{3/2}.
\end{align*}
Here, the appearance of
\[\big\|\lambda_\pm^2\lambda_\mp^2j_{\pm}\big\|_{L^\infty_x}\lesssim \left\|\lambda_\pm^2\lambda_\mp^2j_{\pm}\right\|_{H^2_x}\]
also necessitates that we close our total energy estimates at least at third order.

\textbf{Case 2}:  $|\beta|=1$. In this case, we have either $s=0$ or $s=1$. If $s=0$, we have
\begin{align*}
	\mathbf{J} &\lesssim\int_{0}^{t}\int_{\Sigma_{\tau}}\lambda_\pm^4\lambda_\mp^4|j^{[1]}_{\pm}| |\nabla z_\mp||\nabla z^{[1]}_\pm|dxd\tau\\
	&\lesssim\left\|\lambda_\pm\lambda_\mp j^{[1]}_{\pm}\right\|_{L^2_{[0,t]} L^2_x}\left\|\lambda_\mp^2\lambda_\pm^2\nabla z_\mp\right\|_{L^\infty_\tau L^\infty_x}\left\|\lambda_\pm\lambda_\mp\nabla z^{[1]}_\pm\right\|_{L^2_{[0,t]} L^2_x}\\
	&\lesssim\left\|\lambda_\pm\lambda_\mp j^{[1]}_{\pm}\right\|_{L^2_{[0,t]} L^2_x}\left\|\lambda_\mp^2\lambda_\pm^2\nabla z_\mp\right\|_{L^\infty_\tau H^2_x}\left\|\lambda_\pm\lambda_\mp\nabla z^{[1]}_\pm\right\|_{L^2_{[0,t]} L^2_x}\\
	&\lesssim   \max_{\tau\in[0,t]} \mathcal{E}(\tau)^{3/2};
\end{align*}
while if   $s=1$, we have
\begin{align*}
	\mathbf{J} &\lesssim\int_{0}^{t}\int_{\Sigma_{\tau}}\lambda_\pm^4\lambda_\mp^4|j^{[1]}_{\pm}| |\nabla z^{[1]}_\mp||\nabla z_\pm|dxd\tau\\
	&\lesssim\left\|\lambda_\pm\lambda_\mp j^{[1]}_{\pm}\right\|_{L^2_{[0,t]} L^2_x}\left\|\lambda_\mp\lambda_\pm\nabla z^{[1]}_\mp\right\|_{L^2_{[0,t]} L^2_x}\left\|\lambda_\pm^2\lambda_\mp^2\nabla z_\pm\right\|_{L^\infty_\tau L^\infty_x}\\
	&\lesssim\left\|\lambda_\pm\lambda_\mp j^{[1]}_{\pm}\right\|_{L^2_{[0,t]} L^2_x}\left\|\lambda_\mp\lambda_\pm\nabla z^{[1]}_\mp\right\|_{L^2_{[0,t]} L^2_x}\left\|\lambda_\pm^2\lambda_\mp^2\nabla z_\pm\right\|_{L^\infty_\tau H^2_x}\\
	&\lesssim   \max_{\tau\in[0,t]} \mathcal{E}(\tau)^{3/2}.
\end{align*}
In either case, we applied the Sobolev inequality to the lowest derivative term that appeared: for $s=0$, we used 
\begin{align*}
\big\|\lambda^2_\mp\lambda^2_\pm\nabla z_{\mp}\big\|_{L^\infty_x}\lesssim \left\|\lambda_\mp^2\lambda_\pm^2\nabla z_{\mp}\right\|_{H^2_x},
\end{align*}
and for $s=1$, we used 
\begin{align*}
	\big\|\lambda^2_\pm\lambda^2_\mp\nabla z_{\pm}\big\|_{L^\infty_x}\lesssim \left\|\lambda_\pm^2\lambda_\mp^2\nabla z_{\pm}\right\|_{H^2_x},
\end{align*}
either of which would require our total energy estimates to be at least third order.

\textbf{Case 3}:  $|\beta|=2$. There are three sub-cases. We now deal with them one by one.
 
If $s=0$, there holds
\begin{align*}
	\mathbf{J} &\lesssim\int_{0}^{t}\int_{\Sigma_{\tau}}\lambda_\pm^4\lambda_\mp^4|j^{[2]}_{\pm}| |\nabla z_\mp||\nabla z^{[2]}_\pm|dxd\tau\\
	&\lesssim\left\|\lambda_\pm\lambda_\mp j^{[2]}_{\pm}\right\|_{L^2_{[0,t]} L^2_x}\left\|\lambda_\mp\lambda_\pm\nabla z_\mp\right\|_{L^2_{[0,t]} L^\infty_x}\left\|\lambda_\pm^2\lambda_\mp^2\nabla z^{[2]}_\pm\right\|_{L^\infty_\tau L^2_x}\\
	&\lesssim\left\|\lambda_\pm\lambda_\mp j^{[2]}_{\pm}\right\|_{L^2_{[0,t]} L^2_x}\left\|\lambda_\mp\lambda_\pm\nabla z_\mp\right\|_{L^2_{[0,t]} H^2_x}\left\|\lambda_\pm^2\lambda_\mp^2\nabla z^{[2]}_\pm\right\|_{L^\infty_\tau L^2_x}\\
	&
 \lesssim  \max_{\tau\in[0,t]} \mathcal{E}(\tau)^{3/2}.
\end{align*}
Here, it is the appearance of \begin{align*}
	\big\|\lambda_\mp\lambda_\pm\nabla z_{\mp}\big\|_{L^2_{[0,t]}L^\infty_x}\lesssim \left\|\lambda_\mp\lambda_\pm\nabla z_{\mp}\right\|_{L^2_{[0,t]}H^2_x}
\end{align*}
that requires our total energy estimates to be at least third order.

If $s=1$, we turn  to use Gagliardo-Nirenberg-Sobolev inequality (Lemma \ref{lemma:GNS}), then there holds
\begin{align*}
	\mathbf{J} &\lesssim \int_{0}^{t}\int_{\Sigma_{\tau}}\lambda_\pm^4\lambda_\mp^4|j^{[2]}_{\pm}| |\nabla z^{[1]}_\mp||\nabla z^{[1]}_\pm|dxd\tau\\
	&\lesssim\left\|\lambda_\pm^2\lambda_\mp^2j^{[2]}_{\pm}\right\|_{L^\infty_\tau L^{3/2}_x}\left\|\lambda_\mp\lambda_\pm\nabla z^{[1]}_\mp\right\|_{L^2_{[0,t]} L^6_x}\left\|\lambda_\pm\lambda_\mp\nabla z^{[1]}_\pm\right\|_{L^2_{[0,t]} L^6_x}\\
	&\lesssim\left\|\lambda_\pm^2\lambda_\mp^2j^{[2]}_{\pm}\right\|_{L^\infty_\tau L^{3/2}_x}\left(\left\|\lambda_\mp\lambda_\pm\nabla z^{[1]}_\mp\right\|_{L^2_{[0,t]} L^2_x}+\left\|\lambda_\mp\lambda_\pm\nabla z^{[2]}_\mp\right\|_{L^2_{[0,t]} L^2_x}\right)\\
	&\ \ \ \ \times\left(\left\|\lambda_\pm\lambda_\mp\nabla z^{[1]}_\pm\right\|_{L^2_{[0,t]} L^2_x}+\left\|\lambda_\pm\lambda_\mp\nabla z^{[2]}_\pm\right\|_{L^2_{[0,t]} L^2_x}\right)\\
	&\lesssim \left(\max_{\tau\in[0,t]}E_{\pm,3/2}^{[2]}(\tau)\right)^{2/3}\left(F_{\mp,2}^{[1]}(t)+F_{\mp,2}^{[2]}(t)\right)^{1/2}\left(F_{\pm,2}^{[1]}(t)+F_{\pm,2}^{[2]}(t)\right)^{1/2}\\
	&\lesssim   \max_{\tau\in[0,t]} \mathcal{E}(\tau)^{5/3}.
\end{align*}
We chose not use the Sobolev inequality on any of $j^{[2]}_{\pm}$, $\nabla z^{[1]}_\mp$, and $\nabla z^{[1]}_\pm$ in this scenario, as it  would require a fourth-order $L^2$-type energy estimate to conclude the total energy estimates. In view of the symmetry between the terms $\nabla z^{[1]}_\mp$ and $\nabla z^{[1]}_\pm$, bounding them via the same method could help us to confine the  energy estimate to the lowest possible order. Therefore, 
 we instead employed the Gagliardo-Nirenberg-Sobolev inequality on both $\nabla z^{[1]}_\mp$ and $\nabla z^{[1]}_\pm$, which allows us to restrict the total energy estimate to no greater than third order.
However,  to balance the indices when using the H\"older inequality, we need a bound on $$\left\|\lambda_\pm^2\lambda_\mp^2j^{[2]}_{\pm}\right\|_{L^{3/2}_x}.$$  
This indeed provides our initial motivation for introducing the $L^{3/2}$-type weighted energy norm $E_{\pm,3/2}^{[2]}$ for the term $j^{[2]}_{\pm}$.

If $s=2$, there holds
\begin{align*}
	\mathbf{J} &\lesssim\int_{0}^{t}\int_{\Sigma_{\tau}}\lambda_\pm^4\lambda_\mp^4|j^{[2]}_{\pm}| |\nabla z^{[2]}_\mp||\nabla z_\pm|dxd\tau\\
	&\lesssim\left\|\lambda_\pm\lambda_\mp j^{[2]}_{\pm}\right\|_{L^2_{[0,t]} L^2_x}\left\|\lambda_\mp^2\lambda_\pm^2\nabla z^{[2]}_\mp\right\|_{L^\infty_\tau L^2_x}\left\|\lambda_\pm\lambda_\mp\nabla z_\pm\right\|_{L^2_{[0,t]} L^\infty_x}\\
	&\lesssim\left\|\lambda_\pm\lambda_\mp j^{[2]}_{\pm}\right\|_{L^2_{[0,t]} L^2_x}\left\|\lambda_\mp^2\lambda_\pm^2\nabla z^{[2]}_\mp\right\|_{L^\infty_\tau L^2_x}\left\|\lambda_\pm\lambda_\mp\nabla z_\pm\right\|_{L^2_{[0,t]} H^2_x}\\
	&\lesssim    \max_{\tau\in[0,t]} \mathcal{E}(\tau)^{3/2}.
\end{align*} 
Here, we have used 
\begin{align*}
	\big\|\lambda_\pm\lambda_\mp\nabla z_{\pm}\big\|_{L^2_{[0,t]}L^\infty_x}\lesssim \left\|\lambda_\pm\lambda_\mp\nabla z_{\pm}\right\|_{L^2_{[0,t]}H^2_x},
\end{align*}
 which also requires our total energy estimates to be at least third order.

In summary, we always have
\begin{equation*}
	\mathbf{J}	\lesssim  \max_{\tau\in[0,t]} \mathcal{E}(\tau)^{3/2}.
\end{equation*}
Summing up the resulting estimates for all $0\leq |\beta|\leq 2$,  
we can conclude that  
\begin{align*}
	&\ \ \ \ \sum_{+,-}\sum_{k=0}^{ 2}	E^{[k]}_{\pm,2}(t)+\sum_{+,-}\sum_{k=0}^{2}F^{[k]}_{\pm,2}(t)\\
	&\lesssim \sum_{+,-}\sum_{k=0}^{ 2}E^{[k]}_{\pm,2}(0) +   \max_{\tau\in[0,t]} \mathcal{E}(\tau)^{3/2}+\frac{1}{R}\int_{0}^{t}\Big(\sum_{+,-}\sum_{k=0}^{ 2}	E^{[k]}_{\pm,2}(\tau)+\sum_{+,-}\sum_{k=0}^{ 2}F^{[k]}_{\pm,2}(\tau)\Big)d\tau.\stepcounter{equation}\tag{\theequation}\label{eq:HIGH}
\end{align*}

 \subsection{$L^{3/2}$-type  highest-order (third-order) energy estimates} 
In order to close the energy estimates, we also need to bound $E_{\pm,3/2}^{[2]}$, which appears in the last subsection.   
This subsection is dedicated to bounding this new highest-order energy norm. 

Let us apply the $L^{3/2}$-type energy estimates  to the higher-order system \eqref{eq:DcurlMHD-2}. Taking 
\[f_{\pm}=j^{[2]}_{\pm},\ \ \ \  \rho_{\pm}=\rho_\pm^{[2]}
\]
in \eqref{eq:linear EE-2} leads us to
\begin{align*}
\sum_{+,-}	E^{[2]}_{\pm,3/2}(t)
&\lesssim \sum_{+,-}E^{[2]}_{\pm,3/2}(0) +\sum_{+,-}\int_{0}^{t}\int_{\Sigma_{\tau}}\lambda_\pm^3\lambda_\mp^3 \big|j^{[2]}_{\pm}\big|^{1/2}\big|\rho_\pm^{[2]}\big|+  \frac{1}{R}\int_{0}^{t}\sum_{+,-}	E^{[2]}_{\pm,3/2}(\tau)d\tau\\
&\lesssim \sum_{+,-}E^{[2]}_{\pm,3/2}(0) +\sum_{+,-}\sum_{s\leq 2}\underbrace{\int_{0}^{t}\int_{\Sigma_{\tau}}\lambda_\pm^3\lambda_\mp^3 \big|j^{[2]}_{\pm}\big|^{1/2}\big|\nabla z_\mp^{[s]}\big|\big|\nabla z_\pm^{[2-s]}\big|}_{\mathbf{K}}+  \frac{1}{R}\int_{0}^{t}\sum_{+,-}	E^{[2]}_{\pm,3/2}(\tau)d\tau.
\end{align*}

We turn to bound $\mathbf{K}$ by using H\"older inequality, 
Sobolev inequality (Lemma \ref{lemma:Sobolev}, Lemma \ref{lemma:flux-Linfty}),   Gagliardo-Nirenberg-Sobolev inequality (Lemma \ref{lemma:GNS}), and the spacetime estimates (Lemma \ref{lemma:special}).
Now we also have three sub-cases. In every sub-case, the adjustment of the exponents in the H\"older inequality is achieved through the method of undetermined coefficients. Finally, we find that this adjustment  always relies on the identity  $\frac{1}{3}+\frac{1}{6}+\frac{1}{2}=1$, and moreover the term $|j^{[2]}_{\pm}|^{1/2}$ is always addressed based on the following fact:
\begin{equation*}
\left\|\left(\lambda_\pm\lambda_\mp|j^{[2]}_{\pm}|^{1/2}\right)(\tau,\cdot)\right\|_{ L^3_x}=\left\|\left(\lambda_\pm^2\lambda_\mp^2j^{[2]}_{\pm}\right)(\tau,\cdot)\right\|_{  L^{3/2}_x}^{1/2}= \left(E_{\pm,3/2}^{[2]}(\tau)\right)^{1/3}. 
\end{equation*}
 Let us see these three sub-cases one by one. 

If $s=0$, there holds
\begin{align*}
	\mathbf{K} &\lesssim\int_{0}^{t}\int_{\Sigma_{\tau}}\lambda_\pm^3\lambda_\mp^3|j^{[2]}_{\pm}|^{1/2} |\nabla z_\mp||\nabla z^{[2]}_\pm|dxd\tau\\
	&\lesssim\left\|\lambda_\pm\lambda_\mp|j^{[2]}_{\pm}|^{1/2}\right\|_{L^\infty_\tau L^3_x}\left\|\lambda_\mp\lambda_\pm\nabla z_\mp\right\|_{L^2_{[0,t]} L^6_x}\left\|\lambda_\pm\lambda_\mp\nabla z^{[2]}_\pm\right\|_{L^2_{[0,t]} L^2_x}\\
	&\lesssim\left\|\lambda_\pm^2\lambda_\mp^2j^{[2]}_{\pm}\right\|_{L^\infty_\tau L^{3/2}_x}^{1/2}\left(\left\|\lambda_\mp\lambda_\pm\nabla z_\mp\right\|_{L^2_{[0,t]} L^2_x}+\left\|\lambda_\mp\lambda_\pm\nabla z^{[1]}_\mp\right\|_{L^2_{[0,t]} L^2_x}\right)\left\|\lambda_\pm\lambda_\mp\nabla z^{[2]}_\pm\right\|_{L^2_{[0,t]} L^2_x}\\
	&\lesssim \left(\max_{\tau\in[0,t]}E_{\pm,3/2}^{[2]}(\tau)\right)^{1/3}\left(F_{\mp,2}^{[0]}(t)+F_{\mp,2}^{[1]}(t)\right)^{1/2}\left(F_{\pm,2}^{[2]}(t)\right)^{1/2}\\
	&\lesssim   \max_{\tau\in[0,t]} \mathcal{E}(\tau)^{4/3}.
\end{align*}
 
If $s=1$, there holds 
\begin{align*}
	\mathbf{K} &\lesssim\int_{0}^{t}\int_{\Sigma_{\tau}}\lambda_\pm^3\lambda_\mp^3|j^{[2]}_{\pm}|^{1/2} |\nabla z^{[1]}_\mp||\nabla z^{[1]}_\pm|dxd\tau\\
	&\lesssim\left\|\lambda_\pm\lambda_\mp|j^{[2]}_{\pm}|^{1/2}\right\|_{L^\infty_\tau L^3_x}\left\|\lambda_\mp\lambda_\pm\nabla z^{[1]}_\mp\right\|_{L^2_{[0,t]} L^2_x}\left\|\lambda_\pm\lambda_\mp\nabla z^{[1]}_\pm\right\|_{L^2_{[0,t]} L^6_x}\\
	&\lesssim\left\|\lambda_\pm^2\lambda_\mp^2j^{[2]}_{\pm}\right\|_{L^\infty_\tau L^{3/2}_x}^{1/2}\left\|\lambda_\mp\lambda_\pm\nabla z^{[1]}_\mp\right\|_{L^2_{[0,t]} L^2_x}\left(\left\|\lambda_\pm\lambda_\mp\nabla z^{[1]}_\pm\right\|_{L^2_{[0,t]} L^2_x}+\left\|\lambda_\pm\lambda_\mp\nabla z^{[2]}_\pm\right\|_{L^2_{[0,t]} L^2_x}\right)\\
	&\lesssim \left(\max_{\tau\in[0,t]}E_{\pm,3/2}^{[2]}(\tau)\right)^{1/3}\left(F_{\mp,2}^{[1]}(t)\right)^{1/2}\left(F_{\pm,2}^{[1]}(t)+F_{\pm,2}^{[2]}(t)\right)^{1/2}\\
	&\lesssim   \max_{\tau\in[0,t]} \mathcal{E}(\tau)^{4/3}.
\end{align*}

If $s=2$, there holds
\begin{align*}
	\mathbf{K} &\lesssim\int_{0}^{t}\int_{\Sigma_{\tau}}\lambda_\pm^3\lambda_\mp^3|j^{[2]}_{\pm}|^{1/2} |\nabla z^{[2]}_\mp||\nabla z^{(0)}_\pm|dxd\tau\\
	&\lesssim\left\|\lambda_\pm\lambda_\mp|j^{[2]}_{\pm}|^{1/2}\right\|_{L^\infty_\tau L^3_x}\left\|\lambda_\mp\lambda_\pm\nabla z^{[2]}_\mp\right\|_{L^2_{[0,t]} L^2_x}\left\|\lambda_\pm\lambda_\mp\nabla z^{(0)}_\pm\right\|_{L^2_{[0,t]} L^6_x}\\
	&\lesssim\left\|\lambda_\pm^2\lambda_\mp^2j^{[2]}_{\pm}\right\|_{L^\infty_\tau L^{3/2}_x}^{1/2}\left\|\lambda_\mp\lambda_\pm\nabla z^{[2]}_\mp\right\|_{L^2_{[0,t]} L^2_x}\left(\left\|\lambda_\pm\lambda_\mp\nabla z^{(0)}_\pm\right\|_{L^2_{[0,t]} L^2_x}+\left\|\lambda_\pm\lambda_\mp\nabla z^{[1]}_\pm\right\|_{L^2_{[0,t]} L^2_x}\right)\\
	&\lesssim \left(\max_{\tau\in[0,t]}E_{\pm,3/2}^{[2]}(\tau)\right)^{1/3}\left(F_{\mp,2}^{[2]}(t)\right)^{1/2}\left(F_{\pm,2}^{[0]}(t)+F_{\pm,2}^{[1]}(t)\right)^{1/2}\\
	&\lesssim  \max_{\tau\in[0,t]} \mathcal{E}(\tau)^{4/3}.
\end{align*}

Consequently, we obtain 
\begin{equation}\label{eq:HIGH2}
	\sum_{+,-}	E^{[2]}_{\pm,3/2}(t)
	\lesssim \sum_{+,-}E^{[2]}_{\pm,3/2}(0) + \max_{\tau\in[0,t]} \mathcal{E}(\tau)^{4/3}+  \frac{1}{R}\int_{0}^{t}\sum_{+,-}	E^{[2]}_{\pm,3/2}(\tau)d\tau.
\end{equation}
In this step, no additional terms arise, and we are now prepared to conclude the energy estimates.

\subsection{Improved small weighted energy estimate and lifespan of solutions}

The last step is to complete the energy estimates. 

We first put \eqref{eq:LOW}, \eqref{eq:HIGH} and \eqref{eq:HIGH2} together to derive
\begin{align*}  
	  \mathcal{E}(t) 	\leq C_1\left(   \mathcal{E}(0)+  \max_{\tau\in[0,t]} \mathcal{E}(\tau)^{4/3} +\frac{1}{R}\int_0^t  \mathcal{E}(\tau) d\tau\right),
\end{align*}
where $C_1>1$ is a universal constant. 
Then using Gronwall inequality gives rise to 
\begin{align*}	 
	\mathcal{E}(t) 
		&\leq C_1 \left( \mathcal{E}(0)+\max_{\tau\in[0,t]} \mathcal{E}(\tau)^{4/3}\right)\mathrm{e}^{C_1 t/R}\\
		&\leq C_1(R^{8(1+\delta)/3}\varepsilon+2^{4/3}\varepsilon^{4\kappa/3})e^{C_1 t/R},\ \ \forall t\in[0,T_1].
\end{align*}

A sufficient condition for\begin{align*}	 
	\mathcal{E}(t) \leq \varepsilon^\kappa, \ \ \forall t\in[0,T_1],
\end{align*}
is that for all $t\in[0,T]$, we have
\begin{align*}	 
\begin{cases}
&C_1R^{8(1+\delta)/3}\varepsilon^{1-\kappa}e^{C_1 t/R}\leq 1/2,\\
&C_12^{4/3}\varepsilon^{\kappa/3}e^{C_1 t/R}\leq 1/2,
\end{cases}
\end{align*}
that is,
\begin{align*}	 
\begin{cases}
&T\leq-\frac{R}{C_1}\left(\ln(2C_1)+(1-\kappa)\ln \varepsilon+\frac{8}{3}(1+\delta)\ln R\right)=:F(R),\\
&T\leq -\frac{R}{C_1}\ln\left(2^{7/3}C_1\varepsilon^{\kappa/3}\right).
\end{cases}
\end{align*}
By setting $F'(R^*)=0$, we obtain
\begin{align*}	 
1<R^*=e^{-1}(2C_1)^{-\frac{3}{8(1+\delta)}}\varepsilon^{-\frac{3(1-\kappa)}{8(1+\delta)}}, 
\end{align*}
and 
\begin{align*}	 
\begin{cases}
&T\leq F(R^*)=\frac{8(1+\delta)}{3C_1}e^{-1}(2C_1)^{-\frac{3}{8(1+\delta)}}\varepsilon^{-\frac{3(1-\kappa)}{8(1+\delta)}},\\
&T\leq -\frac{R^*}{C_1}\ln\left(2^{7/3}C_1\varepsilon^{\kappa/3}\right).
\end{cases}
\end{align*}

Finally, we summarize that 
\[	\cmE(t)\leq\mathcal{E}(t) \leq \varepsilon^\kappa, \ \ \forall t\in[0,T],\]
where 
\[T=O(\varepsilon^{-\frac{3}{8(1+\delta)}(1-\kappa)}).\]
We have now completed the proof of Theorem \ref{thm}.

 \appendix
 
 \section{Flow estimate}\label{sec:flow}

 Let $\psi_{\pm}(t,y)=(\psi^1_{\pm}(t,y), \psi^2_{\pm}(t,y), \psi^3_{\pm}(t,y))$ represent the flow mapping from $\Sigma_0$ to $\Sigma_t$, which is  generated by $z_\pm \pm B_0$, i.e. for all $t\in [0,t^*]$ and $y\in \mathbb{R}^3$:
 \begin{equation*}
 	\frac{d}{dt}\psi_{\pm}(t,y)=\left(z_\pm \pm B_0\right)(t,\psi_{\pm}(t,y)), \ \ \psi_{\pm}(0,y)=y.
 \end{equation*}
Integrating  
the equation yields 
 \begin{equation}\label{flow in integration}
 	\psi_{\pm}(t,y)  =y \pm t B_0+\int_0^tz_{\pm}(\tau,\psi_{\pm}(\tau,y))d\tau,
 \end{equation}
where  $y$ serves as the initial label and $x$ denotes the present label. 
Denote the differential of $\psi_{\pm}(t,y)$ at $y$ by $\frac{\partial\psi_{\pm}(t,y)}{\partial y}$. 
Hence we have
\begin{equation*}
	\frac{\partial \psi_\pm(t,y)}{\partial y}=\mathrm{I}+\int_0^t(\nabla z_\pm)(\tau,\psi_\pm(\tau,y))\frac{\partial\psi_\pm(\tau,y)}{\partial y}d\tau.
\end{equation*}
Direct computations result in
\begin{equation*}
	\frac{d}{dt}\Big|\frac{\partial\psi_\pm(t,y)}{\partial y}-\mathrm{I}\Big|\leq\big|(\nabla z_\pm)(t,\psi_\pm(t,y))\big|\Big|\frac{\partial\psi_\pm(t,y)}{\partial y}-\mathrm{I}\Big|+\big|(\nabla z_\pm)(t,\psi_\pm(t,y))\big|.
\end{equation*}
Using Gronwall's inequality, we obtain
\begin{align*}
	\Big|\frac{\partial\psi_\pm(t,y)}{\partial y}-\mathrm{I}\Big|&\leq\exp\Big(\int_0^t\big|(\nabla z_\pm)(\tau,\psi_\pm(\tau,y))\big|d\tau\Big)\underbrace{\int_0^t\big|(\nabla z_\pm)(\tau,\psi_\pm(\tau,y))\big|d\tau}_{A_\pm}=e^{A_\pm}A_\pm.\stepcounter{equation}\tag{\theequation}\label{Apm definition}
\end{align*}

Similar to \eqref{definitionforu}, we also define  $x_i^\pm=x_i^\pm(t,x)$ with $i=1,2$:
\begin{equation}\label{definitionforxi}
	\begin{cases}
		&L_+ x_i^+ = 0,\\
		&x_i^+\big|_{\Sigma_0} = x_i,
	\end{cases}\ \ \ \ \ \ 
	\begin{cases}
		&L_- x_i^- = 0,\\
		&x_i^-\big|_{\Sigma_0} = x_i.
	\end{cases}
\end{equation}
It is clear that $x_3^\pm=u_\pm$ if we let $i=3$ in \eqref{definitionforxi}. 
Since $x^\pm(t,\psi_\pm(t,y))=y$, then 
by series expansion, it follows that 
\begin{align*}
	\frac{\partial x^\pm}{\partial x}\Big|_{x=\psi_\pm(t,y)}
	&=\Big(\mathrm{I}-\Big(\mathrm{I}-\Big(\frac{\partial x^\pm}{\partial x}\Big|_{x=\psi_\pm(t,y)}\Big)^{-1}\Big)\Big)^{-1}\\
	&=\sum_{k=0}^\infty\Big(\mathrm{I}-\Big(\frac{\partial x^\pm}{\partial x}\Big|_{x=\psi_\pm(t,y)}\Big)^{-1}\Big)^k\\
	&=\mathrm{I}+\sum_{k=1}^\infty\Big(\mathrm{I}-\frac{\partial\psi_{\pm}(t,y)}{\partial y}\Big)^k,
\end{align*}
and hence we obtain
\begin{align*}
	\Big|\frac{\partial x^\pm}{\partial x}\Big|_{x=\psi_\pm(t,y)}-\mathrm{I}\Big|
	&\leq \Big|\mathrm{I}-\frac{\partial\psi_{\pm}(t,y)}{\partial y}\Big|\cdot \sum_{k=1}^\infty\Big|\mathrm{I}-\frac{\partial\psi_{\pm}(t,y)}{\partial y}\Big|^{k-1}\\
	&=\Big|\mathrm{I}-\frac{\partial\psi_{\pm}(t,y)}{\partial y}\Big|\cdot\frac{1}{1-\Big|\mathrm{I}-\frac{\partial\psi_{\pm}(t,y)}{\partial y}\Big|}\\
	&\leq \frac{e^{A_\pm}A_\pm}{1-e^{A_\pm}A_\pm}.
\end{align*}

We turn to bound the integration $A_\pm$. 	By the standard Sobolev inequality in $\mathbb{R}^3$, we note that
\begin{align*}
	|\nabla z_\pm(\tau,\psi_\pm(\tau,y))|\lesssim\frac{\left\|\left(\lambda_\pm^2\lambda_\mp^2\nabla z_{\pm}\right)(\tau,\cdot)\right\|_{H^2_x}}{(\lambda_\pm^2\lambda_\mp^2)\big|_{x=\psi_\pm(\tau,y)}}.
\end{align*}
It is easy to switch the variable $\tau$ to $u_\mp$ in the integration $A_\pm$ to get 
\begin{align*}
	A_\pm
	&\lesssim \max_{\tau}\left\|\left(\lambda_\pm^2\lambda_\mp^2\nabla z_{\pm}\right)(\tau,\cdot)\right\|_{H^2_x}\sup_{\tau}\frac{1}{\big|\frac{d}{d\tau}u_\mp(\tau,\psi_\pm(\tau,y))\big|}\int_{\mathbb{R}} \frac{1}{(\lambda_\pm^2\lambda_\mp^2)}du_\mp\\
	&\lesssim \max_{\tau}\left\|\left(\lambda_\pm^2\lambda_\mp^2\nabla z_{\pm}\right)(\tau,\cdot)\right\|_{H^2_x},
\end{align*}
where we have used  $\displaystyle \int_{\mathbb{R}}\frac{1}{(\lambda_\pm^2\lambda_\mp^2)}du_\mp\lesssim 1$. 
When $\left\|\left(\lambda_\pm^2\lambda_\mp^2\nabla z_{\pm}\right)(\tau,\cdot)\right\|_{H^2_x}$ is taken sufficiently small,  we infer that
\begin{align*} 
&	 e^{A_\pm}A_\pm
\leq\frac{1}{2}.
\end{align*}
Thus we can summarize 
\[	\Big|\frac{\partial x^\pm}{\partial x}\Big|_{x=\psi_\pm(t,y)}-\mathrm{I}\Big|\leq  1.\]

 \section{Separation of weight}\label{sec:weight}

 	Thanks to \eqref{flow in integration}, 
 $u_\pm=x_3^\pm$ and $x^\pm(t,\psi_\pm(t,y))=y$, we obtain 
 \begin{equation*}
 	u_\pm(t,\psi_\pm(t,y))=x_3^\pm(t,\psi_\pm(t,y))=y_3=\psi_{\pm}^3(t,y)\mp t-\int_0^tz_{\pm}^3(\tau,\psi_{\pm}(\tau,y))d\tau,
 \end{equation*}
that is,
 \begin{equation*}
 	u_\pm(t,x)=x_3\mp t-\int_0^tz_{\pm}^3(\tau,\psi_{\pm}(\tau,\psi^{-1}_{\pm}(t,x)))d\tau.
 \end{equation*}
Using the standard Sobolev inequality in $\mathbb{R}^3$, we note that
\begin{align*}
	\|z_\pm(\tau,\cdot)\|_{L^\infty}\lesssim\frac{\left\|\left(\lambda_\pm^2\lambda_\mp^2\nabla z_{\pm}\right)(\tau,\cdot)\right\|_{H^2_x}}{\lambda_\pm^2\lambda_\mp^2}\leq \frac{\left\|\left(\lambda_\pm^2\lambda_\mp^2\nabla z_{\pm}\right)(\tau,\cdot)\right\|_{H^2_x}}{R^{\frac{4(1+\delta)}{3}}}.
\end{align*}
When $\left\|\left(\lambda_\pm^2\lambda_\mp^2\nabla z_{\pm}\right)(\tau,\cdot)\right\|_{H^2_x}$ is taken sufficiently small, this gives
\begin{equation*}
		\|z_\pm(\tau,\cdot)\|_{L^\infty}\leq \frac{1}{2}. 
\end{equation*} 
Thus it follows that 
 \begin{equation*}
 	\big|(u_--u_+)-2t\big|\leqslant\int_0^t\big(\|z_+^3\|_{L^\infty}+\|z_-^3\|_{L^\infty}\big)d\tau\leqslant t,
 \end{equation*}
 which leads to the estimates $	t\leqslant |u_+-u_-|\leqslant 3t$,  $|u_+|+|u_-|\geqslant t$, and at least one of the inequalities $|u_+|\geqslant\frac{t}{2}$ and  $|u_-|\geqslant\frac{t}{2}$ hold. Thus we obtain  
 \begin{equation*}
\left(\lambda_\pm^4\lambda_\mp^4\right)(t,\cdot)=(R^2+|u_+|^2)^{\frac{2(1+\delta)}{3}}(R^2+|u_-|^2)^{\frac{2(1+\delta)}{3}}
 	\geq  R^{\frac{4(1+\delta)}{3}}(R^2+\tfrac{t^2}{4})^{\frac{2(1+\delta)}{3}}\geq  \tfrac{1}{2}R^{\frac{4(1+\delta)}{3}}(R^2+t^2)^{\frac{2(1+\delta)}{3}},
 \end{equation*}
which yields \eqref{eq:separation weight} immediately.


\begin{thebibliography}{99}	

	
	
\bibitem{Alfven} Alfv\'en, Hannes. Existence of electromagnetic-hydrodynamic waves. Nature, 150 (1942),  405-406.  
	
	
\bibitem{Bardos-Sulem-Sulem} Bardos, Claude W.; Sulem, Catherine; Sulem, Pierre-Louis. Longtime dynamics of a conductive fluid in the presence of a strong magnetic field. Trans. Amer. Math. Soc.305 (1988), no.1, 175-191.  
	
	
	
\bibitem{Cai-Lei} Cai, Yuan; Lei, Zhen. Global well-posedness of the incompressible magnetohydrodynamics. 
	Arch. Ration. Mech. Anal. 228 (2018), 969-993. 
	
	
\bibitem{Cobb-Fanelli}	Cobb, Dimitri; Fanelli, Francesco. Elsässer formulation of the ideal MHD and improved lifespan in two space dimensions. 	J. Math. Pures Appl. (9) 169 (2023), 189-236.
	

	
\bibitem{Dai}	Dai, Mimi. 	Almost sure existence of global weak solutions for supercritical electron MHD. Discrete Contin. Dyn. Syst. Ser. B 29 (2024), no. 3, 1584-1610.  

\bibitem{Du-Zhang} Du, Lihuai; Zhang, Ting. 
Almost sure existence of global weak solutions for incompressible MHD equations in negative-order Sobolev space. 
J. Differential Equations 263 (2017), no. 2, 1611-1642.  

\bibitem{Fefferman-Mc-R-R-2014}
Fefferman, Charles L.; McCormick, David S.; Robinson, James C.; Rodrigo, Jose L. 
Higher order commutator estimates and local existence for the non-resistive MHD equations and related models. 
J. Funct. Anal. 267 (2014), no. 4, 1035-1056. 

\bibitem{Fefferman-Mc-R-R-2018} 
Fefferman, Charles L.; McCormick, David S.; Robinson, James C.; Rodrigo, Jose L. 
Local existence for the non-resistive MHD equations in nearly optimal Sobolev spaces. 
Arch. Ration. Mech. Anal. 223 (2017), no. 2, 677-691.

\bibitem{He-Xu-Yu} He, Ling-Bing; Xu, Li; Yu, Pin. 
On global dynamics of three dimensional magnetohydrodynamics: nonlinear stability of Alfvén waves. 
Ann. PDE 4 (2018), no. 1, Paper No. 5, 105 pp.

\bibitem{Li-Yu} Li, Mengni; Yu, Pin. On the rigidity from infinity for nonlinear Alfvén waves. J. Differential Equations 283 (2021), 163-215.

\bibitem{Li} Li, Mengni. 
Inverse scattering of Alfvén waves in three dimensional ideal magnetohydrodynamics. Adv. Math. 435 (2023), Paper No. 109363, 69 pp. 

\bibitem{Lin-Xu-Zhang}
Lin, Fanghua; Xu, Li; Zhang, Ping. 
Global small solutions of 2-D incompressible MHD system. 
J. Differential Equations 259 (2015), no. 10, 5440-5485.

\bibitem{Lin-Zhang} Lin, Fanghua; Zhang, Ping. Global small solutions to an MHD-type system: the three-dimensional case. Comm. Pure Appl. Math.  67 (2014), 531-580. 

 
 

\bibitem{Liu-Xin-1} Liu, Sicheng; Xin, Zhouping.          
Local well-posedness of the incompressible current-vortex sheet problems. 
Adv. Math. 475 (2025), Paper No. 110339.

 




\bibitem{Wei-Zhang-2017} Wei, Dongyi; Zhang, Zhifei. Global well-posedness of the MHD equations in a homogeneous magnetic field. 
Anal. PDE 10 (2017), 1361-1406.  


\bibitem{Wei-Zhang-2018} Wei, Dongyi; Zhang, Zhifei. Global well-posedness of the MHD equations via the comparison principle.  
Sci. China Math. 61 (2018), 2111-2120. 

\bibitem{Xie-Yang} Xie, Feng; Yang, Tong. Lifespan of solutions to MHD boundary layer equations with analytic perturbation of general shear flow. 
Acta Math. Appl. Sin. Engl. Ser. 35 (2019), no. 1, 209-229. 

 

\bibitem{Xu-Zhang} Xu, Li; Zhang, Ping. 
Global small solutions to three-dimensional incompressible magnetohydrodynamical system. 
SIAM J. Math. Anal. 47 (2015), no. 1, 26-65.
\end{thebibliography}
\end{document}